\documentclass[12pt]{amsart}
\usepackage{amsfonts}
\usepackage{bbm}
\usepackage{mathrsfs}
\usepackage{color}
\usepackage{esint}
\usepackage{amssymb,amsmath,amsfonts,amsthm,dsfont,mathtools}
\usepackage{amssymb,enumerate}
\usepackage{amsmath}
\usepackage[colorlinks=true, pdfstartview=FitV, linkcolor=blue, citecolor=blue, urlcolor=blue]{hyperref}

\numberwithin{equation}{section}

\newtheorem{thm}{Theorem}[section]
\newtheorem{lem}{Lemma}[section]
\newtheorem{cor}{Corollary}[section]

\theoremstyle{definition}
\newtheorem{defn}{Definition}[section]
\newtheorem{rem}{Remark}[section]

\newcommand{\al}{\alpha}

\newcommand{\ga}{\gamma}
\newcommand{\Ga}{\Gamma}

\newcommand{\De}{\Delta}
\newcommand{\la}{\lambda}
\newcommand{\La}{\Lambda}
\newcommand{\om}{\omega}
\newcommand{\Om}{\Omega}

\def\D{\mathscr{D}}
\def\M{\mathcal{M}}
\def\I{\mathscr{I}}
\def\S{\mathbb{S}}
\def\ss{\mathscr{S}}
\def\R{\mathbb R}

\def\Z{\mathbb Z}
\def\N{\mathbb N}
\def\Q{\mathcal Q}

\def\B{\wt{B}}
\def\w{\mathrm \om}
\def\E{\mathcal{E}}
\def\H{\mathcal{H}}
\def\qq{\widetilde{Q}}
\def\hh{\widetilde{Q}}
\def\lv{\mathcal{LV}_q(\A_{t}^G)}
\def\svq{\mathcal{SV}_q(\A_{t}^G)}
\def\sv{\mathcal{SV}_2(\A_{t}^G)}
\def\1{\mathds{1}}
\def\ll{\mathscr{L}}

\def\wt{\widetilde}
\def\({\left(}
\def\){\right)}
\def\A{\mathbf{A}}
\def\sub{\substack}

\def\L{\mathcal{L}}
\def\g{\mathcal{G}}
\def\d{{\rm d}}

\def\ep{\epsilon}
\def\dd{\widetilde{\mathscr{D}}}
\allowdisplaybreaks

\begin{document}

\title[Variational inequalities for bilinear averaging operators]{Variational Inequalities for Bilinear Averaging Operators over Convex Bodies}

\author[Y. Ding, G. Hong, X. Wu]
{Yong Ding, Guixiang Hong, Xinfeng Wu$^*$}
\bigskip

\address{
School of Mathematical Sciences\\
Laboratory of Mathematics and Complex Systems (BNU)\\
Ministry of Education, China\\
Beijing Normal University\\
Beijing 100875, China}
\email{dingy@bnu.edu.cn}

\address{School of Mathematics and Statistics\\
{and Hubei Key Laboratory of Computational Science\\ Wuhan University\\}
Wuhan 430072, The People's Republic of China}
\email{guixiang.hong@whu.edu.cn}

\address{Department of Mathematics\\
China University of Mining \& Technology\\
Beijing 100083, The People's Republic of China}
\email{wuxf@cumtb.edu.cn}

\thanks{*Corresponding author}


\begin{abstract}

We study $q$-variation inequality for bilinear averaging operators over convex bodies $(G_t)_{t>0}$ defined by
\begin{align*}
\A_t^G(f_1,f_2)(x) & =\frac{1}{|G_t|}\int_{G_t} f_1(x+y_1)f_2(x+y_2)\, dy_1\, dy_2, \quad x\in \R^d.
\end{align*}
where  $G_t$ are the dilates of a convex body $G$ in $\R^{2d}$.
We prove that
$$\|V_q(\A_t^G(f_1,f_2): t>0) \|_{L^p} \lesssim \|f_1\|_{L^{p_1}} \|f_2\|_{L^{p_2}},$$
for $2<q<\infty$, $1<p_1,p_2\le \infty$, $1/2<p<\infty$ with $1/p=1/p_1+1/p_2$. The target space $L^p$ should be replaced by $L^{p,\infty}$ for $p_1=1$ and/or $p_2=1$, and by dyadic BMO when $p_1=p_2=\infty$.
As applications, we obtain variational inequalities for bilinear discrete averaging operators, bilinear averaging operators of Demeter-Tao-Thiele, and ergodic bilinear averaging operators. As a byproduct, we also obtain the same mapping properties for a new class of bilinear square functions involving conditional expectation, which are of independent interest.
\end{abstract}

\maketitle

\section{Introduction}

For $1\le q<\infty$, the $q$-variation seminorm (or $V_q$ seminorm) of a family $(a_t(x))_{t>0}$ of complex-valued functions is defined as
$$
V_q(a_t(x): t>0)=\sup \bigg\{(\sum_{k\ge 0} |a_{t_{k+1}}(x)-a_{t_{k}}(x)|^q)^{\frac{1}{q}}\bigg\},
$$
where the supremum runs over all increasing sequences $(t_k)_{k\ge 0}$ of positive numbers.
The $q$-variation inequality is a crucial tool
in studying the pointwise convergence for a family of operators, due to the fact that it immediately implies the pointwise convergence of the underlying
family of operators without using the Banach principle via the corresponding maximal
inequality.
Moreover, the variational inequality is stronger than the maximal inequality, in the following sense:
$$\sup_{t>0} |a_t(x)| \le |a_{t_0}(x)| + 2V_q(a_t(x): t>0)$$
for any $t_0>0$, and hence is of more interest.

The first $q$-variation inequality was proved by L\'{e}pingle \cite{L} for martingales.
Bourgain \cite{B} proved variational inequalities for the ergodic averages, and directly deduced pointwise convergence results.
Bourgain's work \cite{B} has inaugurated a new research direction in harmonic analysis
and ergodic theory. In \cite{CJRW,CJRW2,JKRW,JRW,JRW2},
Jones and his collaborators systematically studied variational inequalities for ergodic averages and truncated singular integrals of homogeneous type. Since then  many mathematicians established variational inequalities for various operators in ergodic theory and harmonic analysis   (see  \cite{JSW,JW,DMT,OSTTW,MT1,MT2,M,LX, HLP,MTX,KZ,LW,HM,CDHL,MT,MST,MST2,MTZ,BMSW,BMSW2,Z}, among others).

 The purpose of this article is to  establish variational inequalities for a class of multilinear averaging operators over convex bodies. This is the first $q$-variation inequality in the multi-linear Calder\'{o}n-Zygmund theory.
 In the linear case, similar results on averaging operators were studied first by Bourgain \cite{B}, and subsequently by Jones et al \cite{JKRW, JRW2}. The weighted case was investigated by Ma-Torrea-Xu \cite{MTX,MTX2}. The dimension-free variational inequalities for averaging operators over convex bodies were recently established by Bourgain-Mirek-Stein-Wr\'{o}bel in \cite{BMSW}. For relevant maximal inequalities, we refer the reader to the works of Bourgain \cite{B1,B2} and Carbery \cite{C}.

The present article is concerned with variational inequalities in the multilinear setting.
For $t>0,$ let $\wt{Q}_t=Q_t\times Q_t$ be the cube in $\R^{2d}$ centered at origin of side length $t$. The bilinear averaging operators $\A_t^{\widetilde{Q}}$ over cubes $\widetilde{Q}$ are just the tensor product of two linear averaging operators $A_t^Q$:
\begin{align*}
\A_t^{\widetilde{Q}}(f_1,f_2)(x) & =\left(\frac{1}{|Q_t|}\int_{Q_t} f_1(x+y_1) dy_1\right)\left( \frac{1}{|Q_t|}\int_{Q_t} f_2(x+y_2)\, dy_2\right)\\
&=: A_t^Q(f_1)(x)\cdot A_t^Q(f_2)(x),\quad x\in \R^d.
\end{align*}
Applying the trivial estimate
\begin{align}\label{eq:1.4}
V_q(a_t\cdot b_t: \, t>0) \le \Big(\sup_{t>0} |a_t|\Big)\cdot V_q(b_t:\, t>0)+ \Big(\sup_{t>0}|b_t|\Big)\cdot V_q(a_t:\, t>0),
\end{align}
the desired variational inequalities for $\A_t^{\widetilde{Q}}$ follow easily from known variational inequalities for $A_t^Q$ (see \cite{JRW,HM}) and the Hardy-Littlewood maximal inequality.
However, for multilinear averaging operators over balls (and more generally, over convex bodies), the variational inequalities are not so simple, and cannot be deduced from results in the linear case. Moreover, in contrast to the case of maximal operators, the $q$-variation seminorms are not monotone, i.e., $a_t\le b_t$ for all $t>0$ does \textit{not} imply $V_q (a_t:\,t>0)\le V_q (b_t:\,t>0)$. Hence there is no a priori comparison between $q$-variations of averaging operators over cubes and the ones over balls (or convex bodies), and estimates for one average operator do not imply the same estimates for the other. Therefore, it would be interesting and non-trivial to establish variational inequalities for multilinear averaging operators over convex bodies, which is the main goal of the current paper.

To state our main result, we need to recall some definitions. For notational simplicity, in this article, we shall concentrate on the bilinear case.
Let $G$ be a non-empty convex body in $\R^{2d}$, which means that $G$ is a bounded convex open subset of $\R^{2d}$.
For $t>0$ set
$$G_t=\{(x,y)\in \R^{2d}:\, t^{-1}\cdot (x,y)\in G\}.$$
For $f,g\in L^1_{loc}(\R^d),$  the bilinear averaging operator associated to $G_t$ is defined by
\begin{align*}
\mathbf{A}_t^G(f_1,f_2)(x) & = \frac{1}{|G_t|}\int_{G_t} f_1(x+y_1)f_2(x+y_2)\, dy_1\, dy_2,\quad x\in \R^d.
\end{align*}
The dyadic BMO space, $BMO_\d$, is defined via the seminorm:
$$
\|f\|_{BMO_\d}  \simeq \sup_{Q \, \mathrm{dyadic}} \inf_{a_Q} \,  \frac{1}{|Q|}\int_Q |f(x)-a_Q|\, dx.
$$

The main result of this article is the following:
\begin{thm}\label{thm:main1}
 Let $2 < q < \infty$. Then the following statements are true:
\begin{enumerate}
  \item [\rm (i)] For $1/2<p<\infty$, $1<p_1,p_2\le\infty$ satisfying $1/p=1/p_1+1/p_2$ there exists a constant $C=C(p_1,p_2,q,d)$ such that
  for any $f_1 \in L^{p_1}(\R^d)$, $f_2\in L^{p_2}(\R^d)$,
\begin{equation*}
 \|V_q(\A_t^G(f_1,f_2):t>0)\|_p \le C \|f_1\|_{p_1}\, \|f_2\|_{p_2};
\end{equation*}
For either $p_1=\infty$ or $p_2=\infty,$ $L^\infty$ should be replaced by $L^\infty_c.$
  \item [\rm(ii)]  For any $\al>0$, $1\le p_2\le \infty$, $1/2\le p \le 1$
satisfying $1+1/p_2=1/p$,  there exists a
  constant $C=C(p_2,q,d)$ such that for any $f_1 \in L^1(\R^d)$, $f_2\in L^{p_2}(\R^d)$,
\begin{equation*}
|\{V_q(\A_t^G(f_1,f_2):t>0) > \al\}|  \le  \frac{C}{\al^p} \|f_1\|_1^p\, \|f_2\|_{p_2}^p;
\end{equation*}
When $p_2=\infty,$ $L^\infty$ should be replaced by $L^\infty_c.$
  \item [\rm(iii)] If $p_1=p_2=\infty,$ then there exists a constant $C=C(d)$ such that
  for any $f_1,\, f_2 \in  L^\infty_c (\R^d)$ (the space of all compactly supported $L^\infty$ functions),
\begin{equation*}
 \|V_q(\A_t^G(f_1,f_2):t>0)\|_{BMO_{\d}}\le C \|f_1\|_\infty \|f_2\|_\infty.
\end{equation*}
\end{enumerate}
\end{thm}

\begin{rem}
(i) In contrast with the fact that linear variational inequalities are Banach valued,  Theorem \ref{thm:main1} contains both Banach and quasi-Banach valued inequalities (that is, the target space exponent $p$ could be smaller than 1). The constraint $q>2$ coincides with the one in the linear case.

(ii) The operator $V_q(\A_t^G)$ is not of
restricted weak type $(\infty,\infty, \infty),$ that is, the $L^\infty \times L^\infty \rightarrow  L^\infty$ variational inequality fails, even for characteristic functions of measurable sets of finite measures, as will be shown via a counterexample in Section \ref{Sec6} below.
In this sense, the $L^\infty \times L^\infty \rightarrow BMO_\d$ estimate in part (iii) cannot be strengthened.
\end{rem}

 Our strategy of the proof can be described as follows.  We first establish the $L^\infty\times L^2\rightarrow L^2$ variation bounds.
We consider the long and short variation operators separately. For the long variation, matters are reduced to showing the $L^\infty\times L^2\rightarrow L^2$ bound for a new bilinear square function:
$$
\L(f_1,f_2)(x)=\left( \sum_{k\in \Z} | \A_{2^k}^G(f_1,f_2)(x)-\E_k f_1(x) \cdot \E_k f_2(x)|^2 \right)^{\frac{1}{2}},
$$
where $\E_k f$ is the conditional expectation with respect to the $\sigma$-algebra generated by dyadic cubes of side length $2^k$ in $\Bbb R^d.$
Such a square function falls outside the multilinear vector-valued Calder\'{o}n-Zygmund theory studied in literature. To achieve our goal, we shall develop several new tools including
bilinear almost orthogonality principles, a paraproduct type decomposition involving conditional expectation,
uniform Carleson measure estimates,
and new bilinear maximal function. These, together with an adaption of the geometric arguments in \cite{JRW} to the multilinear setting, allow us to conclude the $L^\infty\times L^2\rightarrow L^2$ variational estimate.

Next, based on the $L^\infty\times L^2\rightarrow L^2$ variation estimate, we further prove the $L^\infty\times L^\infty \rightarrow BMO_\d$ estimate, which naturally extends the $L^\infty\rightarrow BMO_\d$ estimate for linear averages in \cite{HM}. We also show, by constructing a counterexample, that
the stronger $L^\infty\times L^\infty\rightarrow L^\infty$ variation estimates fails even for characteristic functions of measurable subsets of finite measures. Regarding the $L^1$ endpoints, we prove the $L^1\times L^{p_2}\rightarrow L^{p,\infty}$ variation estimates for all $1\le p_2\le \infty$ by using a variant of multi-linear Calder\'{o}n-Zygmund decomposition in \cite{GT}.

Combining the above mentioned estimates, we obtain that the $q$-variation operator is of restricted weak types
$(1,\infty,1)$, $(\infty, 1, 1)$,  $(1,1,1/2),$  $(\infty, s, s)$ and $(s,\infty,  s)$, for any $2<s<\infty$. Hence the proof of Theorem \ref{thm:main1} are concluded by applying the multi-linear real interpolation of Grafakos-Kalton in \cite{GK} (see also \cite{Grafakos}). As a byproduct of our arguments, we obtain the same estimates for the bilinear square operator $\L$,
which are of independent interest (see Section \ref{Sec7} below).

Finally, we point out that the techniques used here can be extended to the multi-linear context as well as to the weighted setting, and are expected to be useful for other operators such as the multilinear truncated singular integrals studied by Grafakos-Torres \cite{GT1}. We also mention that in \cite{DOP} Do, Oberlin, and Palsson considered a different class of bilinear averaging operators, which are closely related to the bilinear Hilbert transform. Similar bilinear variational inequalities were proved for $2/3<p<\infty$ by using techniques of time-frequency analysis, which are quite different from ours.

This article is organized as follows. In Section \ref{Sec2}, we prove some lemmas which will be used in the proof of Theorem \ref{thm:main1}. Section \ref{Sec3} is devoted to establishing the $L^\infty \times L^2\rightarrow L^2$ bound. We further prove the $L^\infty \times L^\infty \rightarrow BMO_{\d}$ and $L^1\times L^{p_2}\rightarrow L^{p,\infty}$ variational estimates in Sections \ref{Sec4} and \ref{Sec5} respectively. In Section \ref{Sec6}, we first discuss the restricted weak type estimates and then conclude the proof of Theorem \ref{thm:main1}. The same bounds for the bilinear square function are established in Section \ref{Sec7}. Finally, applications to various bilinear averaging operators are presented in Section \ref{Sec8}.
\section{Preliminaries and some lemmas}\label{Sec2}

In this section, we give some lemmas that will be used in the proof of Theorem \ref{thm:main1}.

\begin{lem}[Bilinear almost orthogonality principle] \label{almost}
{Let $(\S_n)_{n\in \Z}$ be a sequence of operators,}
and let $(\sigma(j))_{j\in \Z}$ be a sequence of positive numbers with $w:=\sum_j \sigma(j)<\infty$.
\begin{enumerate}
  \item [\rm (i)] If $\|\S_k (u_n, v_n)\|_{2} \lesssim \sigma (n-k) a(u_n,v_n)$  for any $n,k\in \Z$,
then
$$\sum_k \left\| \sum_n \S_k\left( u_n, v_n \right) \right\|_{2}^2 \le w \cdot  \sum_{n}a(u_n,v_n)^2.$$
  \item [\rm (ii)] If $\|\S_k (u_n, v_n)\|_{2} \le \sigma (n-k) a_{k}(u_n,v_n)$  for any $n,k\in \Z$,
then
$$\sum_k \left\| \sum_n \S_k\left( u_n, v_n \right) \right\|_{2}^2 \le w \cdot  \sup_n\sum_{k}a_{k}(u_{k-n},v_{k-n})^2.$$
\end{enumerate}

\end{lem}

\begin{proof}
We prove (i) first. 
By the triangle inequality in $L^{2}$,
\begin{align*}
 \left\|  \sum_{n} \S_k\left( u_n,  v_n \right)\right\|_{2}\le  \sum_{n} \left\| \S_k\left( u_n,  v_n \right)\right\|_{2} \le  \sum_{n} \sigma(k-n) a(u_n,v_n).
\end{align*}
The desired estimate then follows by taking the $\ell^2$ norm, and  applying Young's inequality for (discrete) convolution of sequences
$$\|a*_{\rm d} b\|_{\ell^{2}}\le \|a\|_{\ell^{2}} \|b\|_{\ell^{1}}, $$
with $a=(a(u_n,v_n))_{n\in \Z}$ and $b=(\sigma(n) )_{n\in \Z}$, where $*_{\rm d}$ means convolution with respect to the counting measure.

The proof of (ii) is similar. By our hypothesis and triangle inequality,
\begin{align*}
 \left\|  \sum_{n} \S_k\left( u_n,  v_n \right)\right\|_{2}\le  \sum_{n} \left\| \S_k\left( u_n,  v_n \right)\right\|_{2} \le  \sum_{n} \sigma(k-n) \|a_{k}(u_n,v_n)\|_2.
\end{align*}
Using Minkowski's inequality, we have
\begin{align*}
\sum_k \left\| \sum_n\S_k\left( u_n, v_n \right) \right\|_{2}^2
&\le  \sum_k \Big(\sum_{n} \sigma(k-n) \|a_{k}(u_n,v_n)\|_2 \Big )^{2} \\
&= \sum_k \Big(\sum_{n'} \sigma(n')\|a_{k}(u_{k-n'},v_{k-n'})\|_2 \Big )^{2} \\
&\le \left(\sum_{n'}  \sigma(n')\cdot (\sum_{k}\|a_{k}(u_{k-n'},v_{k-n'})\|_2^2)^{1/2}\right)^2\\
& \le w^2\cdot \sup_{n}\sum_{k}\|a_{k}(u_{k-n},v_{k-n})\|_2^2,
\end{align*}
concluding the proof.
\end{proof}


For any cube in $\R^d$, denote by $\ell(Q)$ the side-length of $Q.$ 
For $j\in \Z,$ denote by $\mathscr{D}_j$ (resp. $\dd_j$) the set of all dyadic cubes of side-length $2^j$ in $\R^d$ (resp. $\R^{2d}$). The conditional expectation of a locally integrable function $f$ with respect to the increasing family of $\sigma$-algebras generated by $\mathscr{D}_j$ is defined by
 $$
 \E_j f(x)=\sum_{Q\in \mathscr{D}_j } \left(\frac{1}{|Q|}\int_Q f(y)dy \right)\cdot \1_Q(x),\quad \forall \ j\in \Z.
 $$
The dyadic martingale difference operator $d_j$ is defined by $d_j(f)=\E_{j-1}f-\E_j f$.

The following result is a paraproduct type decomposition involving conditional expectation.
\begin{lem}\label{para}
For $k\in\Z$ and $f_1,f_2\in L^{2}$, define bilinear operators $\L_k$ by
$$\L_k(f_1, f_2)(x)=\A_{2^k}^G(f_1,f_2)(x)-\E_k f_1 (x) \cdot \E_k f_2(x).$$ Then, for almost every $x\in \R^d,$
$$\L_k(f_1, f_2)(x)=\sum_{n\in \Z} \L_k( d_{1,n}, \E_{n-1}f_2)(x) + \sum_{n\in \Z} \L_k(\E_{n} f_1, d_{2,n})(x).$$
\end{lem}
\begin{proof}
By the Cauchy-Schwarz inequality, for any given $x\in \R^d,$
$$|\E_j f_i(x)|\le [\E_j|f_i|^2(x)]^{\frac{1}{2}} \le 2^{-jd} \|f_i\|_2 \rightarrow 0,\quad i=1,2,$$
as $j\rightarrow +\infty$. Moreover, for any $(y_1,y_2)\in \R^{2d}$ and any $x\in \R^d,$
$$\sup_j |\E_j f_1(y_1) \E_j f_2(y_2)| \le \M(f_1)(y_1) \M(f_2)(y_2)   \in L^1((x,x)+G_{2^k}).$$
It follows by the Lebesgue dominated convergence theorem that
$$\lim_{j\rightarrow +\infty}\L_k(\E_{j}f_1, \E_{j}f_2)(x)=0,\quad \ a.e. \ x\in \R^d.$$
On the other hand, Lebesgue differentiation theorem yields $\E_j f_i(x)\rightarrow f(x)\ a.e.$ as $j\rightarrow -\infty$,
which, by the Lebesgue dominated convergence theorem, implies
\begin{equation*}\
\L_k(f_1, f_2)(x)= \lim_{j\rightarrow -\infty}\L_k(\E_{j}f_1, \E_{j}f_2)(x),\quad \ a.e.  \ x\in \R^d.
\end{equation*}
Hence, by the bi-linearity of $\L_k,$ we can express $\L_k(f_1, f_2)$ as
\begin{align}\label{new3.3}
\begin{split}
\L_k(f_1, f_2)(x)
&= \lim_{l\rightarrow -\infty}\L_k(\E_{l-1}f_1, \E_{l-1}f_2)(x) -\lim_{j\rightarrow +\infty} \L_k(\E_{j}f_1, \E_{j}f_2)(x)\\
&= \lim_{l\rightarrow -\infty\atop j\rightarrow +\infty} \sum_{n=l}^{j} [\L_k(\E_{n-1}f_1, \E_{n-1}f_2)(x) - \L_k(\E_{n}f_1, \E_{n}f_2)(x)]\\
&= \lim_{l\rightarrow -\infty\atop j\rightarrow +\infty} \sum_{n=l}^{j} \Big\{[\L_k(\E_{n-1}f_1, \E_{n-1}f_2)(x) - \L_k(\E_{n}f_1, \E_{n-1}f_2)(x) ]\\
&\qquad \qquad \qquad +[\L_k(\E_{n}f_1, \E_{n-1}f_2)(x) -\L_k(\E_{n}f_1, \E_{n}f_2)(x)]\Big\}\\
&=\sum_{n\in \Z} \L_k( d_{1,n}, \E_{n-1}f_2)(x) + \sum_{n\in \Z} \L_k(\E_{n} f_1, d_{2,n})(x),
\end{split}
\end{align}
where all series converge for  $a.e. \ x\in \R^d$.
\end{proof}

 \begin{lem}  \label{Lem 4.1}
 Let $ b\in BMO(\R^d)$, $\delta_{2^k}(t)$ be the Dirac mass at the point $t=2^k,$ and
let $\zeta(x)=(1+|x|)^{-d-\epsilon}$ for some $\epsilon>0$. Then
 \begin{enumerate}
  \item[\rm (i)]  $(d\mu_n(x,t))_{n\in \N}=(\sum_{k\in\Z} |\E_{k+1-n}b(x)-\E_{k-n}b(x)|^2 \, dx\, \delta_{2^k}(t))_{n\in \N}$
 is a sequence of Carleson measures on $\R_+^{d+1}$ with norm bounded by $C\|b\|_{BMO}^2$ uniformly in $n$.
  \item[\rm (ii)] $(d\nu_n(x,t))_{n\in \N}=\(\sum_{k\in\Z} (\zeta_{k}*|\E_{k+1-n}b-\E_{k-n}b|^2)(x) \, dx\, \delta_{2^k}(t)\)_{n\in \N}$  is a sequence of Carleson measures on $\R_+^{d+1}$ with norm bounded by $C\|b\|_{BMO}^2$ uniformly in $n$. Here $\zeta_k(x)=2^{-kd}\zeta(2^{-k}x).$
\end{enumerate}
 \end{lem}
\begin{proof}

For a cube $Q$ in $\R^d$ we let $Q^*$ be the cube with the same center and orientation whose side-length is $100  \sqrt{d}\, \ell(Q)$, where $\ell(Q)$ denotes the side-length of $Q$. Fix a cube $Q$ in $\R^d,$ split $b$ as
$$
b=(b-b_{2Q})\1_{Q^*}+(b-b_{2Q})\1_{(Q^*)^c}+b_{2Q},
$$
where $b_{2Q}=|2Q|^{-1}\int_{2Q} b(y)\, dy.$
Let $T(Q)=Q\times (0, \ell(Q)).$ Since $\E_{k+1-n}b_{2Q}=\E_{k-n}b_{2Q}=b_{2Q},$
$$\E_{k+1-n}b_{2Q}-\E_{k-n}b_{2Q}=0.$$
Thus,
$$
\mu(T(Q)) :=\sum_{2^k\le \ell(Q)} \int_{Q} |\E_{k+1-n}b-\E_{k-n}b|^2 \, dx\le 2 I_1+2 I_2,
$$
where
\begin{align*}
I_1 &=\sum_{k\in \Z} \int_{\R^d} |\E_{k+1}[(b-b_{2Q})\1_{Q^*}](x)-\E_{k}[(b-b_{2Q})\1_{Q^*}](x)|^2 dx\\
\intertext{and}
I_2 &= \sum_{2^k\le \ell(Q)} \int_{Q} |\E_{k+1-n}[(b-b_{2Q})\1_{(Q^*)^c}](x)-\E_{k-n}[(b-b_{2Q})\1_{(Q^*)^c}](x)|^2 \, dx.
\end{align*}

By the $L^2$ boundedness of martingale square function (see e.g. \cite{JSW}), we have
$$
I_1\lesssim \int_{Q^*} |b(x) -b_{2Q}|^2dx\lesssim |Q| \,\|b\|_{BMO}^2,
$$
where the implicit constant is independent of $n.$

Next, let us show that $I_2=0.$
Indeed, for any $x\in Q$ and for any $k$ with $2^k\le \ell(Q)$,
$$
\E_{k+1-n}[(b-b_{2Q})\1_{(Q^*)^c}](x)=\frac{1}{|Q(k-n,x)|}\int_{Q(k-n,x)\,\cap\, (Q^*)^c} (b(y)-b_{2Q}) \, dy,
$$
where $Q(k-n,x)$ is the unique dyadic cube in $\R^d$ of side-length $2^{k+1-n}$ that contains $x$.
For all $n\in \N,$ $\ell(Q(k-n,x)) \le 2\ell(Q)$, which implies $Q(k-n,x) \cap (Q^*)^c=\varnothing.$   Hence $I_2=0$ for all $n\in \N$.

Altogether, we obtain
$$
\mu(T(Q))\le C |Q| \,\|b\|_{BMO}^2,
$$
where $C$ is independent of $n.$ This means $(\mu_n)_{n\in \N}$ is a collection of Carleson measures with norms at most $C\|b\|_{BMO}^2$, uniformly in $n.$

Let us now prove part (ii). We need to show that, for any cube $Q=Q(x_0,\ell(Q))$ and any $n\in \N,$
$$
\int_{Q} \sum_{2^k\le \ell(Q) } (\zeta_{k}*|\E_{k+1-n}b-\E_{k-n}b|^2)(x) \, dx \lesssim \|b\|_{BMO}^2,
$$
uniformly in $Q$ and $n\in \N$.
Define
\begin{align*}
F(y,2^k)&=\int_{\R^d} \zeta_k(y-z)\1_{Q }(z)\, dz,\qquad \forall\, y\in \Bbb R^d, \, k\in \Bbb Z;\\
F^*(x)&=\sup_{(y,k):\, |x-y|\le 2^k} |F(y,2^k)|, \qquad \forall\, y\in \Bbb R^d.
\end{align*}
Clearly $F(y,2^k)\le \|\zeta\|_{L^1}.$ Moreover, for $|y-x_0|\ge 2\sqrt{d}\, \ell(Q)$ and $z\in Q,$ the triangle inequality implies
$|y-z|\ge |y-x_0|-|z-x_0|\ge |y-x_0|/2;$ therefore
$$F(y,2^k)\le C\frac{\ell(Q)^d}{2^{kd}\left(\frac{|y-x_0|}{2^k} \right)^{d+\epsilon}}\le C\left(\frac{|y-x_0|}{\ell(Q)} \right)^{-d-\epsilon}\le C \left(\frac{|x-x_0|}{\ell(Q)} \right)^{-d-\epsilon}  $$
for all $x$ satisfying $|x-y|\le 2^k,$ where the constants $C$ is independent of $k.$
Hence
$F^*(x)\le C(1+\frac{|x-x_0|}{\ell(Q)})^{-d-\epsilon},$  and therefore $F^*\in L^1(\R^d)$.
By part (i) and the Carleson inequality (see \cite{FS} or \cite{Stein93}), we have
$$
\int_{Q} \sum_{2^k\le \ell(Q) } (\zeta_{k}*|\E_{k+1-n}b-\E_{k-n}b|^2)(x) \, dx
\lesssim \|b\|_{BMO}^2 \|F^*\|_{L^1}\lesssim \|b\|_{BMO}^2,
$$
uniformly in $Q$ and $n$.
\end{proof}

For a locally integrable function $f$ on $\R^d$, the uncentered Hardy-Littlewood maximal function $\M(f)$ is defined as
$$\M (f)(x)=\sup_{B}\frac{1}{|B|}\int_{B}\, f(y)\, dy,$$
where the supremum is taken over all balls $B$ in $\R^d$ which contain $x$.

The following lemma establishes a uniform Carleson measure estimate, which can be viewed as a martingale version of Grafakos-Miyachi-Tomita's result in \cite{GMT} and will be a crucial tool for our proof of the theorem.

\begin{lem}\label{lem: 2.4}
Let $1<l<2$ and $\zeta$ be defined as above. Then there exists a constant $C$ depending on $d,\zeta$ and $l$,   but not on $n$, such that
$$
\sum_{k}\int_{\R^d} (\zeta_k*|f|^l)^{2/l}(x) (\zeta_k* |\E_{k+1-n}b-\E_{k-n}b|^l)^{2/l}(x)\, dx\le C\|f\|_2^2\, \|b\|_{BMO}^2,\quad \forall \, n\in \Bbb N.
$$
\end{lem}
\begin{proof}
 First, the H\"{o}lder inequality gives
\begin{align*}
(\zeta_k* |\E_{k+1-n}b-\E_{k-n}b|^l)^{2/l} &\le \left(\|\zeta\|_{L^1}^{1-l/2} [\zeta_k* |\E_{k+1-n}b-\E_{k-n}b|^2(x)]^{l/2}\right)^{2/l}\\
&\sim (\zeta_k* |\E_{k+1-n}b-\E_{k-n}b|^2)(x),
\end{align*}
which implies
\begin{align*}
&\sum_{k}\int_{\R^d} (\zeta_k*|f|^l)^{2/l} (\zeta_k* |\E_{k+1-n}b-\E_{k-n}b|^l)^{2/l}\, dx\\
&\qquad\lesssim  \sum_{k}\int_{\R^d} (\zeta_k*|f|^l)^{2/l} (\zeta_k* |\E_{k+1-n}b-\E_{k-n}b|^2)(x)\, dx.
\end{align*}
By Lemma \ref{Lem 4.1} (ii), $\sum_{k\in\Z} \zeta_{k}*|\E_{k+1-n}b-\E_{k-n}b|^2(x) dx\, \delta_{t_k}(t)$  is a sequence of Carleson measures on $\R_+^{d+1}$ with norm bounded by $C\|b\|_{BMO}^2$, uniformly in $n$. By Carleson's inequality and the $L^{\frac{2}{l}}$ boundedness of $\M$, the last term above is dominated by a constant multiple of
\begin{align*}
\|b\|_{BMO}^2 \int_{\R^d} \sup_{|z-x|\le 2^k} [(\zeta_k *|f|^l)(x)]^{2/l}dx \lesssim \|b\|_{BMO}^2  \|\M(|f|^l)\|_{L^{2/l}}^{2/l}\lesssim \|b\|_{BMO}^2 \|f\|_{L^2}^2,
\end{align*}
where the implicit constant does not depend on $n$.
This completes the proof.
\end{proof}

The following easy variational inequality was stated in \cite[pp. 13-14]{CDHX}.
\begin{lem}\label{martingale-square} \cite{CDHX}.
For $q> 2$, we have
$$\|V_q(\E_j f_1\cdot \E_j f_2: j\in \Z)\|_{L^2(\R^d)} \lesssim \min(\|f_1\|_{L^{2}(\R^d)} \|f_2\|_{L^{\infty}(\R^d)}, \|f_1\|_{L^{\infty}(\R^d)} \|f_2\|_{L^{2}(\R^d)}).$$
\end{lem}

We now introduce a new maximal function for $k$-measurable functions, which will be used to pointwise dominate the bilinear averages.  Assume that $h_1$ and $h_2$ are $(n-1)$-measurable functions on $\R^d$, which means that they are constant on each atom $Q \in \mathscr{D}_{n-1}.$ For every $Q \in \mathscr{D}_{n-1}$,
denote by $M_i^Q, i=1,2,$ the maximum of $|h_i|$ on $Q$ and the cubes neighbouring $Q$ (i.e., contained in $3Q$). Define the maximal functions $h_i^*$ on $\R^d$ by setting  $h_i^*(x)=M_{i}^Q$, where $Q$ is the unique atom in $\mathscr{D}_{n-1}$ containing $x.$ It was observed in \cite{JRW2} that the maximal function $h_i^*$ dominates the linear ball averages $A_{2^k}(h_i)$ if $k\le n-1$, and is bounded on $L^2$:
$$\int_{\R^d} [h_i^*]^2\lesssim \int_{\R^d} h_i^2.$$

In the bilinear contexts, it is easy to see that $|\A_{2^k}^{G}(h_1,h_2)|$ is smaller than $h_1^*(x)\cdot h_2^*(x)$ when $k<n$ (see \eqref{1.1} below). However, the desired $L^2$ bound
 $$\int_{\R^d} (h_1^*)^2\cdot (h_2^*)^2 \lesssim \int_{\R^d} |h_1|^2\cdot|h_2|^2$$
fails.
Indeed, in 1-dimensional case, let $I_1, I_2$ be two adjacent dyadic intervals, suppose $h_1=1$ in $I_2,$ and $h_1=0$ elsewhere,
 and let $h_2=1$ on $I_1$ and $h_2=0$ elsewhere. Then it is easy to see that the above inequality fails.

  To fix this issue, we observe that the following inequalities hold
 \begin{align*}
 \int_{\R^d} [(h_1^* \cdot |h_2|)^*]^2 &\lesssim \int_{\R^d} [h_1^* \cdot |h_2|]^2 \lesssim \int_{\R^d} [|h_1|\cdot|h_2|]^2,\\
 \int_{\R^d} [(|h_1|\cdot h_2^*)^*]^2 &\lesssim \int_{\R^d} [|h_1|\cdot h_2^*]^2 \lesssim \int_{\R^d} [|h_1|\cdot|h_2|]^2.
 \end{align*}
This observation leads to the following definition of bi-sublinear maximal function:
 $$[h_1,h_2]^+(x)=\max\{(h_1^* \cdot |h_2|)^*, \,  (|h_1|\cdot h_2^*)^*\}.$$
Then
\begin{equation}\label{new3.2}
  \int_{\R^d}\{ [h_1, h_2]^+\}^2 \lesssim \int_{\R^d} |h_1|^2\cdot|h_2|^2.
\end{equation}
Moreover, this maximal function dominates the bilinear averages, which is given by
\begin{lem}\label{lem:max}
Let $k<n$. Assume that $h_1$ and $h_2$ are $(n-1)$-measurable functions on $\R^d$, which means that they are constant-valued  on each atom $Q \in \mathscr{D}_{n-1}.$ Then
$$ |\A_{2^{k}}^G(h_1, h_2)(x)| \le [h_1, h_2]^+(x),\quad \forall \, x \in \R^d.$$
\end{lem}
\begin{proof}
Let $k\le n-1.$ For any $x\in \R^d,$ by \eqref{1.1}, $(x,x)+G_{2^k}\subset (x,x)+\widetilde{B}_{2^k}$. Therefore, $(x,x)+G_{2^k}$ can intersect with at most $2^{2d}$ atoms in $\dd_{n-1}$. Recall that $\dd_j$ denotes the set of all dyadic cubes in $\R^{2d}$ of side-length $2^j$.
We only deal with the case where $(x,x)+G_{2^k}$ intersects with exactly $2^{2d}$ atoms in $\dd_{n-1}$, as the other cases are easier and can be treated in the same way.  We denote these $2^{2d}$ atoms by $\widetilde{Q}_{(I,J)}$, where $(I,J)=(i_1,\ldots, i_d, j_1,\ldots,j_d)\in \Z^d\times \Z^d$ denotes the coordinates of their lower-left corner. Without loss of generality, we may assume that $n=1$ and the lower-left corner of the union of these atoms is the origin so that $i_k, j_\ell\in \{0,1\}$ for all $k,\ell=1,\ldots,d.$
In other words, $(I,J) \in \{0,1\}^d \times \{0,1\}^d.$

Writing
\begin{align*}
|\A_{2^{k}}^G(h_1, h_2)(x)|
 \le \frac{1}{|G_{2^k}|} \sum_{I,J\in \{0,1\}^d} \int_{[(x,x)+G_{2^k}]\, \cap\, \widetilde{Q}_{(I,J)}} |h_1(y_1)h_2(y_2)| dy_1 dy_2,
\end{align*}
we see that it suffices to show the pointwise bound: For every $(I,J)\in \{0,1\}^{2d}$,
$$|h_1(y_1)h_2(y_2)| \le [h_1, h_2]^+(x), \quad \forall \, (y_1,y_2)\in \qq_{(I,J)}=Q_I\times Q_J.$$

A key observation is that, for any $I, J\in \{0,1\}^d$, either $Q_I=Q_J$ or $Q_I$ neighbours $Q_J$; so, in either case,
we have
$|h(z)|\le h^*(w)$ for all $(n-1)$-measurable functions $h$ and all $z\in Q_I$ and $w\in Q_J$.
Hence, for any $(I,J)\in \{0,1\}^{2d}$ and $(y_1,y_2)\in \qq_{(I,J)}=Q_I\times Q_J,$
$$|h_1(y_1)h_2(y_2)|\le (|h_1|\cdot h_2^*)(y_1) \le (|h_1|\cdot h_2^*)^*(x)\le [h_1, h_2]^+(x), $$
concluding the proof.
\end{proof}

Let $S(\Bbb R^d)$ denote the space of functions of the form
$\sum_{i=1}^N c_i \1_{E_i},$ where each measurable subset $E_i$ of $\R^d$ has finite measure.
The following lemma is a multilinear version of Marcinkiewicz interpolation theorem with initial restricted weak type conditions (cf. \cite[Theorem 4.6]{GK}),
 which will be used to prove Theorem \ref{thm:main1}.
\begin{lem} \cite{GK,Grafakos} \label{interpolation}
Let $m$ be a positive integer, and let $T$ be a multi-sublinear
operator defined on $S(\R^d)\times \cdots \times S(\R^d)$ taking values in the set of measurable functions on $\R^d$.
For $1\le k \le m+1$ and $1\le j\le m$ we are given $p_{k,j}$,
with $0<p_{k,j}\le \infty$ and $0<q_k\le \infty.$
Suppose that the open convex hull of the points
$$\vec{P_k}=\Big(\frac{1}{p_{k,1}},\cdots, \frac{1}{p_{k,m}}\Big)$$
is an open set in $\R^d,$ and $T$ is of restricted weak types
$(p_{k,1},\ldots, p_{k,m}, q_k)$ for all $1\le k\le m+1$, that is,
$$\|T(\1_{E_1},\ldots, \1_{E_m})\|_{L^{q_k,\infty}} \le \theta_k \prod_{j=1}^m |E_j|^{\frac{1}{p_{k,j}}}$$
for all $1\le k\le m+1$ and for all subsets $E_j$ of $\R^d$
with $|E_j|<\infty.$
Let
$$\vec{P}=\Big(\frac{1}{p_1},\cdots, \frac{1}{p_m}\Big)=\sum_{k=1}^{m+1}\eta_k \vec{P}_k$$
for some $\eta_k\in (0,1)$ such that $\sum_{k=1}^{m+1} \eta_k =1,$ and define $1/q=\sum_{k=1}^{m+1} \eta_k/ q_k.$
If $\ga_j\not=0$ for all $j=1,\ldots, m$ and $1/q\le 1/p_1+\cdots+1/p_m,$ then $T$ is bounded from $L^{p_1}\times \cdots \times L^{p_m}$
 to $L^q,$ and moreover,
$$\|T(f_1,\ldots, f_m)\|_{L^q} \lesssim \Big(\prod_{k=1}^{m+1} \theta_k^{\eta_k}\Big)\Big(\prod_{j=1}^m \|f_j\|_{L^{p_j}}\Big)$$
for all $f_j\in L^{p_j},\, j=1,\ldots, m.$
\end{lem}
 \section{$L^\infty \times L^2 \rightarrow L^2$ estimate}\label{Sec3}
{Since $V_q(\A_t^G(f_1,f_2):t>0)=V_q(\A_t^{G_{t_0}}(f_1,f_2):t>0)$ for any $t_0>0,$ we may assume throughout that $G$ is normalized such that
\begin{equation}\label{1.1}
  \B_\tau \subset G \subset \B_1 \qquad \text{for some } 0<\tau<1.
\end{equation}}
We shall show that the $L^\infty \times L^2 \rightarrow L^2$ estimate holds when $f_1\in L_c^\infty$ and $f_2\in L^2$.
To proceed, we first divide the $q$-variation into long and short variations. Let $(t_i)_{i\geq0}$ be an increasing sequence.
For each interval $I_i=(t_i,t_{i+1}]$, we consider two cases.
\begin{itemize}
  \item Case 1: $I_i$ does not contain any power of 2;
  \item Case 2: $I_i$ contains powers of 2. In this case, we decompose $I_i$ further as disjoint union of intervals like $(t_i, 2^{m_i}]\cup(2^{m_i},2^{n_i}]\cup(2^{n_i}, t_{i+1}]$ such that $(t_i, 2^{m_i}]$ or $(2^{n_i}, t_{i+1}]$ do not contain any more power of 2.
\end{itemize}
Let
$\ss$ be the set of all ``short intervals'' consisting of all intervals that is contained in $(2^k, 2^{k+1}]$ for some $k\in \Z$, that is the intervals in Case 1 and the ones of the form $(t_i, 2^{m_i}]$ or $(2^{n_i}, t_{i+1}]$ in Case 2.
Let $\ll$ consist of all disjoint intervals of the form $(2^{m_j}, 2^{n_j}]$ in Case 2.

Clearly, $\ss\cup \ll$ is a disjoint family of intervals. We use the following pointwise inequality:
\begin{align*}
&\Big(\sum_{i\geq0} |(\A_{t_{i+1}}^G- \A_{t_{i}}^G) (f_1,f_2) |^q\Big)^{\frac{1}{q}} \\
& \lesssim \sup_{(t_i)_i} \Big(\sum_{I_i\in \ll} |(\A_{t_{i+1}}^G - \A_{t_i}^G)(f_1,f_2) |^q \Big)^{\frac{1}{q}}+ \sup_{(t_i)_i} \Big(\sum_{I_i\in \ss} |(\A_{t_{i+1}}^G- \A_{t_{i}}^G) (f_1,f_2)|^q \Big)^{\frac{1}{q}}\\
&\lesssim V_q(\A^G_{2^k}(f_1,f_2):\,k\in\mathbb Z)+\Big(\sum_{k\in\mathbb Z}V^q_q(\A^G_{t}(f_1,f_2):\,t\in(2^k,2^{k+1}])\Big)^{\frac1q}\\
&=: \lv (f_1,f_2)+ \svq(f_1,f_2).
\end{align*}
We call $\lv$ and $\svq$ long and short variation operators respectively.

The bounds of $\lv$ and $\svq$ will be proved in the following two subsections.

\subsection{Long variation estimates}\label{sub3.1}

To prove the long variation estimates, we shall compare the bilinear averaging operators with conditional expectation,
which reduces matters to bounding a certain bilinear square function.
More specifically,
define a bilinear operator $\L_k$ by
$$\L_k(f_1, f_2)(x)=\A_{2^k}^G(f_1,f_2)(x)-\E_k f_1 (x) \cdot \E_k f_2(x),$$
and the square operator $\L$ by
$$
\L(f_1,f_2)(x)=\bigg( \sum_{k\in \Z}|\L_k(f_1, f_2)(x)|^2 \bigg)^{\frac{1}{2}}.
$$
Then the following estimate holds pointwise
\begin{align*}
\lv(f_1,f_2)& \lesssim  \L(f_1,f_2)+\sup_{(n_i)_i}\Big(\sum_{i} |(\E_{n_{i+1}} - \E_{n_i})(f_1,f_2) |^q \Big)^{\frac{1}{q}}.
\end{align*}
The second term is just the $q$-variation for martingales, for which the desired bound follows from \eqref{eq:1.4} and the known $q$-variation inequalities for martingales (cf. \cite{PX}).
It suffices to establish
\begin{equation*}\label{1}
  \left\|\L(f_1,f_2) \right\|_2\lesssim \|f_1\|_{\infty} \|f_2\|_{2}.
\end{equation*}

Since $f_1,f_2\in L^2(\R^d)$, we may write $f_1=\sum_{n\in \Z} d_{1,n}$ and $f_2=\sum_{m\in \Z} d_{2,m}$,
where $d_{1,n}= d_n(f_1)$ and $d_{2,n}= d_n(f_2)$ are martingale differences and both series converge almost everywhere and in the topology of $L^2$.
Moreover,
$$  \sum_n\|d_{1,n}\|_2^2 = \|f_1\|_2^2,\quad \sum_n\|d_{2,n}\|_2^2=\|f_2\|_2^2.$$
Using Lemma \ref{para},
it thus suffices to show
\begin{align}
\sum_{k\in \Z}  \|\sum_{n>k}\L_k( d_{1,n}, \E_{n-1}f_2)\|_{2} ^2 &\lesssim \|f_1\|_{\infty}^2 \|f_2\|_{2}^2, \label{aim1}\\
\sum_{k\in \Z}  \|\sum_{n\le k} \L_k( d_{1,n}, \E_{n-1}f_2)\|_{2}^2 &\lesssim \|f_1\|_{\infty}^2 \|f_2\|_{2}^2, \label{aim2}\\
\sum_{k\in \Z}  \|\sum_{n>k} \L_k(\E_{n}f_1,  d_{2,n})\|_{2}^2 &\lesssim \|f_1\|_{\infty}^2 \|f_2\|_{2}^2, \label{aim3}\\
\sum_{k\in \Z}  \|\sum_{n\le k} \L_k(\E_{n}f_1,  d_{2,n})\|_{2}^2 &\lesssim \|f_1\|_{\infty}^2 \|f_2\|_{2}^2. \label{aim4}
\end{align}

We first prove \eqref{aim1} and \eqref{aim3}.
We assume $n>k,$ then $\E_k d_{1,n} =d_{1,n}$.
Write
\begin{align*}
&\| \L_k(d_{1,n}, \E_{n-1}f_2)\|_2^2=\sum_{Q\in \mathscr{D}_{n-1}} \int_Q | \L_{k}(d_{1,n}, \E_{n-1}f_2)(x)|^2 \, dx.
\end{align*}
Since $d_{1,n}$ and $\E_{n-1}f_2$ are both constants on the atom $Q\in \mathscr{D}_{n-1}$, we have
$$\L_{k}(d_{1,n}, \E_{n-1}f_2)(x)=\A_{2^k}^G(d_{1,n}, \E_{n-1}f_2)(x)-d_{1,n}(x) \cdot \E_{n-1}f_2(x)=0$$
if $(x,x)+ G_{2^k}\subset Q\times Q.$ Thus, for $x\in Q,$
 $\L_{k}(d_{1,n}, \E_{n-1}f_2)(x)$ is nonzero only if $(x,x)+G_{2^k}$ intersects with the complement of $Q\times Q.$
For any subset $E$ of $\R^{2d}$ containing the origin, denote
$$\H(E, Q)=\{x\in Q:\, [(x,x)+E] \cap (Q\times Q)^c \not= \varnothing\}.$$
By \eqref{1.1},  $\H(G_{2^k}, Q) \subset \H(\widetilde{B}_{2^k}, Q),$ and
$$|\H(G_{2^k}, Q)|\le  |\H(\widetilde{B}_{2^k}, Q)|\lesssim 2^{(d-1)n}\cdot 2^k.$$
From this, the Cauchy-Schwarz inequality and Lemma \ref{lem:max}, it follows that
\begin{align*}
\int_{Q} | \L_{k}(d_{1,n}, \E_{n-1}f_2)|^2 &= \int_{\H(G_{2^k}, Q)} | \L_{k}(d_{1,n}, \E_{n-1}f_2)|^2\\
& \lesssim 2^{(d-1)n}\cdot 2^k \cdot \{[d_{1,n}, \E_{n-1}f_2]^+\}^2 = 2^{-|k-n|} \int_ Q \, \{[d_{1,n}, \E_{n-1}f_2]^+\}^2.
 \end{align*}
Summing this over $Q\in \mathscr{D}_{n-1}$ and using \eqref{new3.2}, we get
\begin{align*}
\| \L_{k}(d_{1,n}, \E_{n-1}f_2) \|_2 &\lesssim 2^{-\frac{|k-n|}{2}} \|[d_{1,n}, \E_{n-1}f_2]^+\|_2 \lesssim 2^{-\frac{|k-n|}{2}} \|d_{1,n}\cdot \E_{n-1} f_2 \|_2.
\end{align*}
Applying {Lemma \ref{almost}, (i) with $\S_k=\L_k$, $u_n=d_{1,n}$, $v_n=\E_{n-1}f_2$, $a(d_{1,n},\E_{n-1}f_2)=\|d_{1,n}\cdot\E_{n-1}f_2\|_2$ and $\sigma(n)=-|n|/2$}, Lemma \ref{Lem 4.1}, (i), and the Carleson inequality, we obtain
 \begin{align*}
\sum_{k} \Big(\sum_{n>k}\|\L_{k}(d_{1,n}, \E_{n-1}f_2) \|_2\Big)^2 &\lesssim \Big(\sum_n 2^{-|n|/2}\Big)^2\cdot \Big(\sum_n \int_{\R^d}  |d_{1,n}|^2\cdot |\E_{n-1} f_2|^{2}\Big)\\
 & \lesssim \|f_1\|_{BMO}^2 \|\M(f_2)\|_2^2\lesssim \|f_1\|_{\infty}^2 \|f_2\|_2^2.
 \end{align*}
This concludes the proof of \eqref{aim1}.

To show \eqref{aim3},
by the same arguments as above, we get
 \begin{align*}
\sum_{k} \Big(\sum_{n>k} \| \L_{k}(\E_{n}f_1, d_{2,n}) \|_2\Big)^2
& \lesssim  \sum_n \int_{\R^d}  |\E_n f_1|^{2}  \cdot |d_{2,n}|^2.
\end{align*}
Now, using the simple estimate $\sup_n\|\E_n f_1\|_\infty\le \|f_1\|_\infty$, the last term above is majorized by
$\|f_1\|_{\infty}^2 \sum_n   \|d_{2,n}\|_2^2 =\|f_1\|_{\infty}^2 \|f_2\|_2^2,$
which gives \eqref{aim3}.

We first prove \eqref{aim2} and \eqref{aim4}. Assume $n\le k.$ Since $\E_k d_{1,n}=0$ in this case,  \eqref{aim2} and \eqref{aim4} will follow from
the following pointwise estimates respectively:  For $1<l<2$,
\begin{equation}\label{eq4.1}
  | \A_{2^k}^G(d_{1,n}, \E_{n-1}f_2) | ^2 \lesssim  2^{-|k-n|(2-2/l)} \A_{2^k}^G(|d_{1,n}|^l, |\E_{n-1}f_2|^l)^{2/l},
\end{equation}
and
\begin{equation}\label{eq4.2}
  | \A_{2^k}^G(\E_{n}f_1, d_{2,n}) | ^2 \lesssim  2^{-|k-n|(2-2/l)} \A_{2^k}^G( |\E_{n}f_1|^l, |d_{2,n}|^l)^{2/l}.
\end{equation}
Indeed, assume that \eqref{eq4.1} and \eqref{eq4.2} hold for the moment, let us prove \eqref{aim2} and \eqref{aim4}. Denote by $Q_{2^k}$ the cube centered at origin and having side-length $2^k$ in $\R^d$ and let $\widetilde \1_{Q_{2^k}}=|Q_{2^k}|^{-1}\1_{Q_{2^k}}$. Recalling that $G_{2^k}\subset \widetilde{B}_{2^{k}}$, we have
\begin{align*}
\A_{2^k}^G(|g_1|, |g_2|)(x)&\lesssim \left(\frac{1}{|Q_{2^k}|}\int_{Q_{2^k}}|g_1(x-y_1)|dy_1\right) \left(\frac{1}{|Q_{2^k}|}\int_{Q_{2^k}}|g_2(x-y_2)|dy_2\right)\\
&=(\widetilde{\1}_{Q_{2^k}} *|g_1|)(x) \cdot (\widetilde{\1}_{Q_{2^k}} *|g_2|)(x) \lesssim(\zeta_k *|g_1|)(x) \cdot (\zeta_k *|g_2|)(x),
\end{align*}
where $\zeta(x)=(1+|x|)^{-d-\epsilon}$ for some $\epsilon>0$.
Applying \eqref{eq4.1} and the above estimate with $g_1=|d_{1,n}|^l$ and $g_2=|\E_{n-1}f_2|^l$, we get
 \begin{align*}
\| \A_{2^k}^G(d_{1,n}, \E_{n-1}f_2) \|_2 &\lesssim   2^{-(1-1/l)|k-n|} \| \A_{2^k}^G(|d_{1,n}|^l, |\E_{n-1}f_2|^l)^{1/l}\|_{2}\\
&\lesssim   2^{-(1-1/l)|k-n|} \|(\zeta_k* |d_{1,n}|^l)^{1/l} \cdot (\zeta_k* |\E_{n-1}f_2|^l)^{1/l}\|_{2}.
 \end{align*}
{Applying Lemma \ref{almost}, (ii) with $a_{k}(u_n,v_n)=\|(\zeta_k* |u_n|^l)^{1/l} \cdot (\zeta_k* |v_n|^l)^{1/l}\|_{2}$, $u_n=d_{1,n}$ and $v_n=\E_{n-1}f_2$}, and Lemma \ref{lem: 2.4} with $b=f_1$ and $f=\M f_2$, we deduce that
 \begin{align*}
&{\sum_{k} \|\sum_{n\le k} \A_{2^k}^G(d_{1,n}, \E_{n-1}f_2) \|_2^2} \\
&\quad\lesssim   \sup_{n'\in \N}\sum_k \int_{\R^d}  (\zeta_k* |d_{1,k-n'}|^l)^{2/l}(x)  (\zeta_k* |\E_{k-n'-1}f_2|^l)^{2/l}(x)\, dx \\
&\quad\lesssim   \sup_{n'\in \N}\sum_k \int_{\R^d}  (\zeta_k* |d_{1,k-n'}|^l)^{2/l}(x)  (\zeta_k* |\M f_2|^l)^{2/l}(x)\, dx \\
&\quad \lesssim \|f_1\|_{BMO}^2 \|\M f_2\|_{2}^2 \lesssim   \|f_1\|_{\infty}^2 \|f_2\|_{2}^2,
 \end{align*}
 which gives  \eqref{aim2}.

For \eqref{aim4}, integrating both sides of \eqref{eq4.2} and using the same arguments as above, we have
 \begin{align*}
& \sum_{k}  \Big\|\sum_{n\le k} \A_{2^k}^G(\E_{n}f_1, d_{2,n}) \Big\|_2^2 \\
&\quad \lesssim \sup_{n'\in \N}\sum_k \int_{\R^d}  (\zeta_k* |d_{2,k-n'}|^l)^{2/l}(x)  (\zeta_k* |\E_{k-n'-1}f_1|^l)^{2/l}(x)\, dx\\
&\quad \lesssim \sup_{n'\in \N}\sum_k \int_{\R^d}  \M( |d_{2,k-n'}|^l)^{2/l}(x)  \cdot \M( \M f_1^l)^{2/l}(x)\, dx\\
&\quad \lesssim \|\M( \M f_1^l)^{2/l}\|_\infty \sup_{n'\in \N}\sum_k \|  \M( |d_{2,k-n'}|^l)\|_{L^{2/l}}^{2/l}\\
&\quad \le  \|f_1\|_{\infty}^2 \cdot \Big(\sup_{n'\in \N}\sum_{k}\|d_{2,k-n'} \|_2^2 \Big)  \lesssim \|f_1\|_\infty^2 \|f_2\|_2^2,
 \end{align*}
which gives \eqref{aim4}.

Let us now show \eqref{eq4.1} and \eqref{eq4.2} under the assumption $n\le k$.
We divide $\R^{2d}$ into all atoms $\hh$ in $\dd_n,$ and write
$$ \int_{(x,x)+G_{2^k}} (d_{1,n}\otimes \E_{n-1} f_2) =\sum_{\hh\in\dd_n} \int_{[(x,x)+G_{2^k}]\, \cap\, \hh} (d_{1,n}\otimes \E_{n-1} f_2).$$
Writing $\hh=Q_1\times Q_2\in \D_n\times \D_n,$  then $\int_{Q_1} d_{1,n}=0$ if $\hh\subset (x,x)+G_{2^k}$, which implies
$$
\int_{\hh} \, (d_{1,n}\otimes \E_{n-1} f_2)  =\left(\int_{Q_1} d_{1,n}\right) \left(\int_{Q_2}\E_{n-1} f_2\right)=0.
$$
Hence $\int_{[(x,x)+G_{2^k}]\, \cap\, \hh} (d_{1,n}\otimes \E_{n-1} f_2)$ is non-zero only if $\hh$ intersects with the boundary of $(x,x)+G_{2^k}.$ Recall that
\begin{equation}\label{set}
I((x,x)+G_{2^k}, n)=\bigcup \, \{[(x,x)+G_{2^k}] \cap \hh:\, \hh\in \dd_n,\, [(x,x)+\partial G_{2^k}] \cap \hh \not=\varnothing\}.
\end{equation}
Then
 \begin{align*}
  \int_{(x,x)+G_{2^k}} (d_{1,n}\otimes \E_{n-1} f_2)
  & =\int_{ I((x,x)+ G_{2^k}, n)}(d_{1,n}\otimes \E_{n-1} f_2)
 \end{align*}
 and
\begin{equation}\label{1.2}
|I((x,x)+G_{2^k}, n)| \lesssim 2^{n-k} 2^{k2d}
\end{equation}
(See \cite{JRW} for more discussions on \eqref{1.2}).

 Using H\"older's inequality, \eqref{1.2} and \eqref{1.1}, we get that
 for $1<l<2$,
 \begin{align*}
  |\A_{2^k}^G(d_{1,n}, \E_{n-1}f_2) | ^l &=\frac{1}{| G_{2^k}|^l} \left|\int_{ I((x,x)+ G_{2^k}, n)}(d_{1,n}\otimes \E_{n-1} f_2)\right|^l\\
  & \lesssim \frac{1}{(2^{2dk})^l} (2^{n+(2d-1)k})^{l-1}  \int_{ (x,x)+ G_{2^k}} (|d_{1,n}|^l\otimes |\E_{n-1} f_2|^l)\\
  &\lesssim 2^{(n-k)(l-1)}\frac{1}{| G_{2^k}|} \int_{ (x,x)+ G_{2^k}} (|d_{1,n}|^l\otimes |\E_{n-1} f_2|^l),
 \end{align*}
 which implies \eqref{eq4.1}. Inequality \eqref{eq4.2} can be proved analogously, the details being omitted.
 This concludes the proof of the $L^\infty \times L^2\rightarrow L^2$ long variation estimates. \\

 \subsection{Short variation estimates}

By the embedding $\ell^2\hookrightarrow \ell^q$, it suffices to bound $\sv$ instead of $\svq$.
Define a bi-subadditive operator $\g_k$   by
 $$\g_k (f_1,f_2)(x)=\Big( \sum_{I_i\in \ss_k} |(\A_{t_{i+1}}^G- \A_{t_i}^G)(f_1,f_2)(x)|^2\Big)^{\frac{1}{2}},$$
where we recall that $I_i \in \ss_k$ means that $I_i=(t_i, t_{i+1}]\subset (2^k, 2^{k+1}]$.
We need to show
 \begin{equation}\label{2}
  \bigg\|\Big(\sum_{k\in \Z} [\g_k(f_1,f_2)]^2\Big)^{1/2} \bigg\|_2\lesssim \|f_1\|_{\infty} \|f_2\|_{2}.
\end{equation}
Similar to \eqref{new3.3}, we can express $(\A_{t_{i+1}}^G- \A_{t_i}^G)(f_1, f_2)$ as
\begin{align}\label{new312}
\begin{split}
&(\A_{t_{i+1}}^G- \A_{t_i}^G)(f_1, f_2)\\
&\qquad = \sum_{n\in \Z} (\A_{t_{i+1}}^G- \A_{t_i}^G)( d_{1,n}, \E_{n-1}f_2) +\sum_{n \in \Z} (\A_{t_{i+1}}^G- \A_{t_i}^G)(\E_{n} f_1, d_{2,n}),
\end{split}
\end{align}
where the series converges pointwise.
Taking the $\ell^2$ norm over $I_i\in \ss_k$ and using triangle inequality, we get
\begin{align*}
\g_k(f_1, f_2)
&\le \sum_{n\in \Z} \g_k ( d_{1,n}, \E_{n-1}f_2) +\sum_{n \in \Z} \g_k  (\E_{n} f_1, d_{2,n}).
\end{align*}

To show \eqref{2}, we consider two cases $n>k$ and $n\le k$ separately, and
it then suffices to show
{
\begin{align}\label{aim5}
\begin{split}
\sum_{k\in \Z}  \|\sum_{n>k+1}\g_k( d_{1,n}, \E_{n-1}f_2)\|_{2}^2 &\lesssim \|f_1\|_{\infty}^2 \|f_2\|_{2}^2, \\
\sum_{k\in \Z}   \|\sum_{n\le k+1} \g_k( d_{1,n}, \E_{n-1}f_2)\|_{2}^2 &\lesssim \|f_1\|_{\infty}^2 \|f_2\|_{2}^2,  \\
\sum_{k\in \Z}    \|\sum_{n>k+1} \g_k(\E_{n}f_1,  d_{2,n})\|_{2}^2 &\lesssim \|f_1\|_{\infty}^2 \|f_2\|_{2}^2,  \\
\sum_{k\in \Z}     \|\sum_{n\le k+1} \g_k(\E_{n}f_1,  d_{2,n})\|_{2}^2 &\lesssim \|f_1\|_{\infty}^2 \|f_2\|_{2}^2.
\end{split}
\end{align}}
We only show the first two inequalities as the others can be handled similarly.

Let us show the first inequality of \eqref{aim5}. We assume $n>k+1$. By the almost orthogonality principle, matters are reduced to showing
 \begin{align}
 \|\g_k(d_{1,n}, \E_{n-1}f_2)\|_2^2 \lesssim 2^{-(n-k)} \|d_{1,n}\cdot \E_{n-1}f_2\|_{2}^2.\label{3.3}
 \end{align}
In fact, once \eqref{3.3} is established,  {it then follows from  Lemma \ref{almost}, (i)
with $\S_k=\g_k$, $u_n=d_{1,n}$, $v_n=\E_{n-1}f_2$ and $a(u_n,v_n)=\|d_{1,n}\cdot \E_{n-1}f_2\|_2$ that
\begin{align*}
&  \sum_{k\in \Z}   \|\sum_{n>k+1} \g_k(d_{1,n}, \E_{n-1}f_2)\|_2  ^2  \lesssim  \sum_n \|d_{1,n}\cdot \E_{n-1}f_2\|_{2}^2 ,
\end{align*}}
which is bounded by $C\|f_1\|_{\infty}^2 \|f_2\|_2^2$, as shown in the proof of \eqref{aim1} and \eqref{aim3}.

Let us show \eqref{3.3}.
 We write
\begin{align*}
 \|\g_k(d_{1,n}, \E_{n}f_2)\|_2^2 & =  \sum_{Q\in \mathscr{D}_{n-1}}  \int_Q \sum_{I_i \in \ss_k} |(\A_{t_{i+1}}^G- \A_{t_i}^G)(d_{1,n}, \E_{n-1}f_2)(x)|^2\, dx.
\end{align*}
We have
 \begin{align*}
&\left( \sum_{I_i \in \ss_k} |(\A_{t_{i+1}}^G-\A_{t_i}^G)(d_{1,n}, \E_{n-1}f_2)|^2\right)^{\frac{1}{2}}(x) \\
& \le \sum_{I_i \in \ss_k} |(\A_{t_{i+1}}^G-\A_{t_i}^G)(d_{1,n}, \E_{n-1}f_2)|(x)\\
& \le \sum_{I_i \in \ss_k} \frac{1}{| G_{t_i}|} \int_{(x,x)+ G_{t_{i+1}}\backslash G_{t_i}} |d_{1,n}|\otimes |\E_{n-1}f_2|\\
&\qquad +\sum_{I_i \in \ss_k} \left(\frac{1}{| G_{t_i}|} -\frac{1}{| G_{t_{i+1}}|}\right)\int_{(x,x)+ G_{t_{i+1}}} |d_{1,n}|\otimes |\E_{n-1}f_2|\\
& \le  \frac{1}{| G_{2^k}|} \sum_{I_i \in \ss_k}\int_{(x,x)+ G_{t_{i+1}}\backslash G_{t_i}} |d_{1,n}|\otimes |\E_{n-1}f_2|\\
&\qquad +\left( \sum_{I_i \in \ss_k} \frac{1}{| G_{t_i}|} -\frac{1}{| G_{t_{i+1}}|}\right)\int_{(x,x)+ G_{2^{k+1}}} |d_{1,n}|\otimes |\E_{n-1}f_2|\\
& \lesssim  \frac{1}{|G_{2^k}|} \int_{(x,x)+ G_{2^{k+1}}} |d_{1,n}|\otimes |\E_{n-1}f_2|\\
& \lesssim \A_{2^{k+1}}^G(|d_{1,n}|, |\E_{n-1}f_2|)(x) \le [d_{1,n},\E_{n-1} f_2]^+(x),
 \end{align*}
where we used Lemma \ref{lem:max} in the last inequality.
For each $Q\in \mathscr{D}_{n-1}$ and $x\in Q$, since $d_{1,n} \otimes \E_{n-1}f_2$ is constant on $Q\times Q$, we see that
 $$\sum_{I_i \in \ss_k} |( \A_{t_{i+1}}^G- \A_{t_i}^G)(d_{1,n}, \E_{n-1}f_2)(x)|^2$$ is nonzero only if,
 for some $I_i\in \ss_k$, the convex body $(x,x)+ G_{t_{i+1}}$ (hence $(x,x)+ G_{2^{k+1}}$) intersects with the complement of $Q\times Q.$
Hence
\begin{align*}
\int_Q \sum_{I_i \in \ss_k} |( \A_{t_{i+1}}^G-\A_{t_i}^G)(d_{1,n}, \E_{n-1}f_2)|^2
& \lesssim \int_{\H( G_{2^{k+1}}, Q)} \{[d_{1,n},\E_{n-1} f_2]^+\}^2 \\
& \lesssim 2^{(d-1)n} 2^k \{[d_{1,n},\E_{n-1} f_2]^+\}^2 \\
& = 2^{k-n} \int_{Q} \{[d_{1,n},\E_{n-1} f_2]^+\}^2,
\end{align*}
where
$$\H(G_{2^{k+1}}, Q)=\{x\in Q:\, [(x,x)+G_{2^{k+1}}] \cap (Q\times Q)^c \not= \varnothing\},$$
and its measure is no more than $2^{(d-1)n} 2^k$.
Summing the above over $Q\in \mathscr{D}_{n-1}$, and using \eqref{new3.2}, we obtain
\begin{align*}
& \sum_{Q\in \mathscr{D}_{n-1}}\int_Q \sum_{I_i \in \ss_k} |( \A_{t_{i+1}}^G- \A_{t_i}^G)(d_{1,n}, \E_{n-1}f_2)|^2 \\
&\qquad \lesssim 2^{k-n}  \int_{\R^d} \, \{[d_{1,n},\E_{n-1} f_2]^+\}^2 \lesssim 2^{k-n}  \int_{\R^d} |d_{1,n}|^2\cdot |\E_nf_2|^2,
\end{align*}
which verifies \eqref{3.3}, and hence the first inequality in \eqref{aim5}.

Finally, let us prove the second inequality in \eqref{aim5}. We assume $k\ge n-1.$ By arguments similar to that in the long variation case, \eqref{aim5} will follow from the pointwise estimate: For $1<l<2,$
 \begin{equation}\label{313}
 \sum_{I_i \in \ss_k} |( \A_{t_{i+1}}^G- \A_{t_i}^G)(d_{1,n}, \E_{n-1}f_2)|^2 \le 2^{(n-k)(2-2/l)} \A_{2^{k+1}}^G (|d_{1,n}|^l, |\E_{n-1}f_2|^l)^{2/l}.
 \end{equation}

To show \eqref{313}, we write
\begin{align*}
&\sum_{I_i \in \ss_k} |(\A_{t_{i+1}}^G- \A_{t_i}^G)(d_{1,n}, \E_{n-1}f_2)|^2\\
 &\lesssim  \sum_{I_i \in \ss_k} \left|\frac{1}{| G_{t_i}|}\int_{(x,x)+ G_{t_{i+1}}\backslash G_{t_i}}(d_{1,n} \otimes \E_{n-1}f_2)\right|^2 \\
 &\qquad + \sum_{I_i \in \ss_k} \left(\frac{1}{| G_{t_i}|}-\frac{1}{| G_{t_{i+1}}|}\right)^2 \left| \int_{(x,x)+ G_{t_{i+1}}} d_{1,n}\otimes \E_nf_2\right|^2\\
 & \lesssim  \frac{1}{| G_{2^k}|^2} \left( \sum_{I_i \in \ss_k} \left|\int_{(x,x)+ G_{t_{i+1}}\backslash G_{t_i}}(d_{1,n} \otimes \E_{n-1}f_2)\right|^l \right)^{2/l} \\
& \qquad + \left(  \sum_{I_i \in \ss_k} \left(\frac{1}{| G_{t_i}|}-\frac{1}{| G_{t_{i+1}}|}\right)^{l} \sup_{I_i\in \ss_k} \left| \int_{(x,x)+ G_{t_{i+1}}} d_{1,n}\otimes \E_nf_2\right|^l \right)^{2/l}\\
  &=:I+II.
\end{align*}

 For the first term $I$, 
 by the fact that  $$
 \int_{\hh} \, (d_{1,n}\otimes \E_{n-1} f_2)  =\left(\int_{Q_1} d_{1,n}\right) \cdot \left(\int_{Q_2} \E_{n-1} f_2\right)=0
 $$ for any $\hh=Q_1\times Q_2\in \dd_n,$ we have
 $$
 \int_{(x,x)+ G_{t_{i+1}}\backslash G_{t_i}}  (d_{1,n}\otimes \E_{n-1} f_2)  =\int_{I((x,x)+ G_{t_{i+1}}\backslash G_{t_i}, n) } (d_{1,n}\otimes \E_{n-1} f_2),
 $$
 where $I((x,x)+ G_{t_{i+1}}\backslash G_{t_i}, n)$ was defined as in \eqref{set}.
From H\"older's inequality, the estimate $|I((x,x)+ G_{t_{i+1}} \backslash G_{t_i}, n)|\lesssim 2^n 2^{(2d-1)k}$ and \eqref{1.1},
it follows that
 \begin{align*}
 I^{l/2} & \lesssim \frac{1}{(2^{2kd})^l}   (2^n2^{(2d-1)k})^{l-1} \sum_{I_i\in \ss_k} \int_{(x,x)+ G_{t_{i+1}}\backslash G_{t_i}} ( |d_{1,n}|^l \otimes |\E_{n-1}f_2|^l) \\
 & \lesssim 2^{(n-k)(l-1)} \frac{1}{| G_{2^{k+1}}|} \int_{(x,x)+ G_{2^{k+1}}} ( |d_{1,n}|^l \otimes |\E_{n-1}f_2|^l),
 \end{align*}
which gives the required bound for $I.$

For the second term $II$, we have
$$
\sum_{I_i\in \ss_k}\left( \frac{1}{| G_{t_i}|}-\frac{1}{| G_{t_{i+1}}| } \right)^{l}\le \left(\frac{1}{| G_{2^k}|}-\frac{1}{| G_{2^{k+1}}| }\right)^{l} \lesssim 2^{-2dkl}.
$$
Using H\"older's inequality, the estimate $|I((x,x)+ G_{t_i}, n)|\lesssim 2^n 2^{(2d-1)k},$ and \eqref{1.1}, we deduce that
\begin{align*}
 II^{l/2} & \lesssim \frac{1}{(2^{2dk})^l} (2^n2^{(2d-1)k})^{l-1} \sup_{I_i\in \ss_k} \int_{(x,x)+ G_{t_{i+1}}} ( |d_{1,n}|^l \otimes |\E_{n-1}f_2|^l) \\
&  \lesssim 2^{(n-k)(l-1)}   \frac{1}{| G_{2^{k+1}}|} \int_{(x,x)+ G_{2^{k+1}}} ( |d_{1,n}|^l \otimes |\E_{n-1}f_2|^l) ,
 \end{align*}
which gives the required bound for $II.$
This verifies \eqref{313} and hence completes the proof of the $L^{\infty}\times L^{2}\rightarrow L^2$ short variation estimate.

\begin{rem} \label{rem3.1}
The hypothesis $f_1\in L_c^\infty$ can be replaced by $f_1\in L^\infty \cap L^2$, and
the $L^\infty \times L^2 \to L^2$ bound still holds true. Interchanging the roles of $f_1$ and $f_2,$
we can conclude that the $L^2\times L^\infty \to L^2$ variational inequality holds for $f_1\in L^2$ and $f_2\in L^\infty \cap L^2.$
\end{rem}
\section{$L^\infty \times L^\infty  \rightarrow BMO_{\d}$ estimate}\label{Sec4}

Let us prove the long variation inequality first, namely, $\lv$ is bounded from $L^\infty \times L^\infty$ to $BMO_{\d}$
 for $f_1,f_2\in L^\infty_c$. 

For $f_i\in L_c^\infty$ and a dyadic cube $Q$ in $\R^d,$ we divide $f_i$ into the local and global parts:
$$f_i=f_i \mathds{1}_{Q^*} +f_i \1_{\R^d\backslash Q^*}=f_i^0+f_i^\infty,$$
where $Q^*=3\sqrt{d}\, Q.$ By the triangle inequality,
\begin{align*}
& \frac{1}{|Q|} \int_Q |\lv(f_1,f_2)-a_Q|  \\
 &\quad \le   \frac{1}{|Q|} \int_Q |\lv(f_1^0,f_2^0)| + \frac{1}{|Q|} \int_Q |\lv(f_1^0,f_2^\infty)| \\
  &\quad \qquad + \frac{1}{|Q|} \int_Q |\lv(f_1^\infty,f_2^0)| + \frac{1}{|Q|} \int_Q |\lv(f_1^\infty, f_2^\infty)-a_Q| \\
 &\quad =: I_1+I_2+I_3+I_4,
\end{align*}
where $a_Q$ is a constant depending on $Q$ which will be determined later.

Since the first three terms $I_1,\, I_2, \, I_3$ can be estimated in the same way, we just estimate $I_2$. Since $f_1^0\in  L^2$ and $f_2^\infty\in L^\infty\cap L^2,$ by the $L^\infty \times L^2\rightarrow L^2$ boundedness of $\lv$ (see Remark \ref{rem3.1}),
\begin{align*}
I_2& \le \left(  \frac{1}{|Q|} \int_Q |V_q(\A_{2^k}^G)(f_1^0, f_2^\infty)|^2 \right)^{\frac{1}{2}} \\
&\lesssim  \frac{1}{\sqrt{|Q|}} \|f_1^0\|_{L^{2}(Q^*)} \|f_2^\infty\|_{L^{\infty} (Q^*)} \le \|f_1\|_\infty
\|f_2\|_{\infty}.
\end{align*}

It remains to estimate $I_4.$ Choosing $a_Q=\lv(f_1^2, f_2^\infty)(c_Q)$, by triangle inequality in $\ell^q$ and the embedding
$\ell^q\hookrightarrow \ell^2,$
it suffices to show that for any $x\in Q,$
\begin{align}\label{new4.1}
\Big(\sum_{k\in \Z} |\A_{2^k}^G (f_1^\infty,f_2^\infty)(x) -\A_{2^k}^G (f_1^\infty,f_2^\infty)(c_Q)|^2 \Big)^{1/2}
 \lesssim \|f_1\|_\infty \|f_2\|_\infty,
\end{align}
where the implicit constant is independent of $x$ and $Q.$

If $2^k<\ell(Q),$ then for any $x\in Q$ and
any $(y,y')\in G_{2^k},$
\begin{align*}
|x+y-(c_Q)|\le |x-(c_Q)|+|y|<\sqrt{d}\, \ell(Q),
\end{align*}
which implies that
$(x,x)+G_{2^k}$ is contained in $Q^* \times Q^*$. Hence, by the support of $f_1^\infty$,
$$
\A_{2^k}^G (f_1^\infty,f_2^\infty)(x)=\A_{2^k}^G (f_1^\infty,f_2^\infty)(c_Q)=0
$$
for all $x\in Q.$

If $2^k\ge \ell(Q)$, then
\begin{align}\label{4.2}
\begin{split}
&|\A_{2^k}^G(f_1^\infty, f_2^\infty)(x)- \A_{2^k}^G(f_1^\infty, f_2^\infty)(c_Q) | \\
&=\frac{1}{| G_{2^k}|} \left|\int_{(x,x)+ G_{2^k}} f_1^\infty\otimes f_2^\infty-\int_{(c_Q,c_Q)+ G_{2^k}} f_1^\infty\otimes f_2^\infty\right|\\
& = \frac{1}{2^{2kd}}|\int_{\R^{2d}} (f_1^\infty\otimes f_2^\infty) (\1_{[(c_Q,c_Q)+ G_{2^k}] \backslash [(x,x)+ G_{2^k}]}-\1_{[(x,x)+ G_{2^k}]\backslash[(c_Q,c_Q)+G_{2^k}]})|\\
& \le  \frac{1}{2^{2kd}} \|f_1\|_{\infty} \|f_2\|_{\infty} \,  |[(c_Q,c_Q)+ G_{2^k}]\, \Delta\, [(x,x)+ G_{2^k}]|\\
&\lesssim \frac{\ell(Q)2^{k(2d-1)}}{2^{2kd}} \|f_1\|_{\infty} \|f_2\|_{\infty},
\end{split}
\end{align}
where $E\,\De\, F$ denotes the symmetric difference of two sets $E$ and $F.$ In the last inequality we used the observation that, by \eqref{1.1},
the measure of projection of the symmetric difference $[(c_Q,c_Q)+ G_{2^k}]\, \Delta\, [(x,x)+ G_{2^k}]$ to a hyperplane perpendicular to the vector $(x,x)-(c_Q,c_Q)$ is not more than the ``volume'' of $(2d-1)$-dimensional ball $B_{2d-1}(0,2^k)$, and therefore
$$
|[(c_Q,c_Q)+ G_{2^k}]\, \Delta\, [(x,x)+ G_{2^k}]|\lesssim  \ell(Q) \cdot 2^{k(2d-1)}.
$$
Summing over $\{k: 2^k\ge \ell(Q)\}$, we obtain
\begin{align*}
&\Big(\sum_{k: 2^k\ge \ell(Q)} |\A_{2^k}^G (f_1^\infty,f_2^\infty)(x) -(\A_{2^k}^G (f_1^\infty,f_2^\infty)(c_Q)|^2 \Big)^{1/2}\\
&\qquad\lesssim \sum_{k: 2^k\ge \ell(Q)} \frac{\ell(Q)}{2^{k}} \|f_1\|_{\infty} \|f_2\|_{\infty}\lesssim \|f_1\|_{\infty} \|f_2\|_{\infty},
\end{align*}
which concludes the proof of the long variation inequality.

Let us now treat the short variation operator. By arguments similar to that given in the long variation case, it suffices to show that,
for any $x\in Q$ and any increasing sequence $(t_i)_i,$
\begin{align*}
\sum_{2^k\ge \ell(Q)} \sum_{ I_i\in \ss_k} & | (\A_{t_{i+1}}^G- \A_{t_{i}}^G)(f_1^\infty, f_2^\infty)(x)
- ( \A_{t_{i+1}}^G - \A_{t_{i}}^G)(f_1^\infty, f_2^\infty)(c_Q)  |^2\lesssim \|f_1\|_\infty^2 \|f_2\|_\infty^2,
\end{align*}
where the implicit constant is independent of $Q$ and the sequence $(t_i)_i.$
By comparing $t_{i+1}-t_i$ with $\ell(Q)$, we divide the summation in $I_i\in \ss_k$ into two pieces according to whether $t_{i+1}-t_i$ is bigger or smaller than $\ell (Q).$

We consider the case $t_{i+1}-t_i> \ell (Q)$ first.
Since the family $(G_{t_i})_i$ of convex bodies are nested, we have that for any $z\in Q$,
\begin{align*}
&| \A_{t_{i+1}}^G(f_1^\infty, f_2^\infty)(z)- \A_{t_{i}}^G(f_1^\infty, f_2^\infty)(z) |\\
&= \left| \left(\frac{1}{| G_{t_{i+1}}|}-\frac{1}{| G_{t_i}|} \right)\int_{(z,z)+ G_{t_i}} f_1^\infty\otimes f_2^\infty+\frac{1}{|G_{t_{i+1}}|} \int_{(z,z)+[G_{t_{i+1}} \backslash G_{t_i}]} f_1^\infty\otimes f_2^\infty \right|\\
&\le  \left(\frac{1}{|G_{t_i}|}-\frac{1}{|G_{t_{i+1}}|} \right)\int_{(z,z)+ G_{t_i}} |f_1^\infty\otimes f_2^\infty|+\frac{1}{| G_{t_{i+1}}|} \int_{(z,z)+[G_{t_{i+1}} \backslash G_{t_i}]} |f_1^\infty\otimes f_2^\infty|\\
&\le  \left(\frac{1}{|G_{t_i}|}-\frac{1}{|G_{t_{i+1}}|} \right) |G_{t_i}| \|f_1\|_\infty \| f_2\|_\infty +\frac{1}{|G_{t_{i+1}}|} (|G_{t_{i+1}}| - |G_{t_i}|) \|f_1\|_\infty \| f_2\|_\infty\\
&= 2 (|G_{t_{i+1}}| - |G_{t_i}|) \frac{1}{|G_{t_{i+1}}|}  \|f_1\|_\infty \| f_2\|_\infty\\
&\lesssim \|f_1\|_\infty \| f_2\|_\infty \int_{|G_{t_i}|}^{|G_{t_{i+1}}|} \frac{1}{u}du.
\end{align*}
Applying the Cauchy-Schwarz inequality and using $t_{i+1}-t_i< \ell(Q)$ and $t_i\sim t_{i+1}$, we can continue the above estimates by
\begin{align*}
& \lesssim \|f_1\|_\infty \| f_2\|_\infty  (|G_{t_{i+1}}| - |G_{t_i}|)^{1/2}  \left(\int_{|G_{t_i}|}^{|G_{t_{i+1}}|} \frac{1}{u^2}du\right)^{1/2}\\
& \lesssim \|f_1\|_\infty \| f_2\|_\infty  \ell(Q)^{1/2} t_i^{d-1/2}  \left(\int_{|G_{t_i}|}^{|G_{t_{i+1}}|} \frac{1}{u^2}du\right)^{1/2}.
\end{align*}
Hence
\begin{align*}
&\bigg(\sum_{2^k\ge \ell (Q)} \sum_{\sub{I_i\in \ss_k\\ t_{i+1}-t_i <\ell (Q)}} | \A_{t_{i+1}}^G(f_1^\infty, f_2^\infty)(x)- \A_{t_i}^G(f_1^\infty, f_2^\infty)(x)|^2\bigg)^{1/2} \\
& \lesssim \|f_1\|_\infty \| f_2\|_\infty  \ell(Q)^{1/2} \left(\sum_{2^k\ge \ell(Q)} 2^{k(2d-1)}  \sum_{I_i\in \ss_k} \int_{|G_{t_i}|}^{|G_{t_{i+1}}|} \frac{1}{u^2}du\right)^{1/2}\\
&\le \|f_1\|_\infty \| f_2\|_\infty  \ell(Q)^{1/2} \left(\sum_{2^k\ge \ell(Q)} 2^{k(2d-1)}  \int_{(\tau 2^{k})^{2d}}^{2^{2(k+1)d}} \frac{1}{u^2}du\right)^{1/2}\\
&\le \|f_1\|_\infty \| f_2\|_\infty  \ell(Q)^{1/2} \left(\sum_{2^k\ge \ell(Q)} \tau^{-2d} 2^{-k}  \right)^{1/2}\\
&\lesssim \|f_1\|_\infty \| f_2\|_\infty,
\end{align*}
where $\tau$ is the constant in \eqref{1.1}.

It remains to treat the case $t_{i+1}-t_i\ge \ell (Q)$.
By an argument similar to \eqref{4.2}, we can get
$$
|\A_{t_{i}}^G (f_1^\infty, f_2^\infty)(x) - \A_{t_{i}}^G (f_1^\infty, f_2^2)(c_Q) |\lesssim |Q|^2 \|f_1\|_\infty \|f_2\|_\infty \frac{1}{|G_{t_i}|}.
$$
Since the number of intervals $I_i \in \ss_k$ satisfying $t_{i+1}-t_i \ge \ell(Q)$ is not more than $C2^k \slash \ell(Q),$
we conclude that
\begin{align*}
II_2 & =\bigg(\sum_{2^k\ge \ell (Q)} \sum_{\sub{I_i\in \ss_k\\ t_{i+1}-t_i \ge \ell (Q)}} | \A_{t_{i+1}}^G(f_1^\infty, f_2^\infty)(x)-\A_{t_{i+1}}^G(f_1^\infty, f_2^\infty)(c_Q)|^2\bigg)^{1/2}\\
&\lesssim |Q|^2 \|f_1\|_\infty \|f_2\|_\infty \bigg(\sum_{2^k \ge \ell(Q)} \sum_{\substack{I_i\in \ss_k \\ t_{i+1}-t_i \ge \ell(Q)}} \frac{1}{|G_{t_i}|^2} \bigg)^{1/2}\\
& \lesssim |Q|^2 \|f_1\|_\infty \|f_2\|_\infty \bigg(\sum_{2^k \ge \ell(Q)} \frac{2^k}{\ell(Q)} \frac{1}{2^{4kd}} \bigg)^{1/2}
 \lesssim \|f_1\|_\infty \|f_2\|_\infty.
\end{align*}
This completes the proof of the short variational inequality.
\section{$L^1 \times L^{p_2}\rightarrow L^{p,\infty}$ estimates, $1+{1}/{p_2}=1/p$} \label{Sec5}
This section is devoted to the proof of $L^1\times L^{p_2}\rightarrow L^{p,\infty}$ estimates for $1\le p_2\le \infty, 1/2\le p\le1$ with $1+1/p_2=1/p$. We distinguish two cases $1  \leq p_2<\infty$ and $p_2=\infty$, as the Calder\'{o}n-Zygmund decomposition does not hold for  functions in $L^\infty$. In the following two subsections, we treat the two cases $1 \leq  p_2<\infty$ and $p_2=\infty$ respectively.

\subsection{$L^1 \times L^{p_2}\rightarrow L^{p,\infty}$ estimates with $1\leq p_2<\infty$}
Fix an $\al$ and assume $f_1\in L^{1}$ and $f_2\in L^{p_2}$, both with compact supports.  Without loss of generality, we may assume $\|f_1\|_1=\|f_2\|_{p_2}=1$. Set $p_1=1$. For $i=1,2$, we perform a variant of Calder\'on-Zygmund decomposition given in \cite{GT} for $L^{p_i}$ functions $f_i$ at height $ \al^{p/p_i}$
to obtain ``good'' and ``bad'' functions $g_i$ and $b_i$, and families of cubes $\{Q_{i,j}\}_j$ with disjoint interiors such that
\begin{enumerate}
  \item[(i)]  $f_i=g_i+b_i,$
  \item[(ii)] $b_i=\sum_j b_{i,j}$,
  \item[(iii)] supp~$(b_{i,j})\subset Q_{i,j},$
   \item[(iv)] $\int b_{i,j}(x)\, dx=0,$
    \item[(v)] $\|b_{i,j}\|_{p_i}^{p_i} \le 2^{d+p_i}  \al^{p} |Q_{i,j}|,$
     \item[(vi)] $\bigcup_j |Q_{i,j}|\le C\al^{-p} \|f_i\|_{p_i}^{p_i},$
      \item[(vii)] $\|b_i\|_{p_i} \le 2^{\frac{d+p_i}{p_i}} \|f_i\|_{p_i},$
       \item[(viii)] $\|g_i\|_{p_i} \le  \|f_i\|_{p_i},$ $\|g_i\|_\infty\le 2^{\frac{d}{p_i}}\al^{\frac{p}{p_i}}.$
\end{enumerate}
Indeed, to show the above decomposition, one can imitate the same idea of classical Calder\'{o}n-Zygmund decomposition, but select a cube $Q$
when $(\frac{1}{|Q|}\int_Q |f_i|^{p_i})^{1/p_i}>\al^{p/p_i};$ the functions $g_i$ and $b_i$ are defined as in the classical case.
Now, let us begin to treat the long and short variation operators separately.
\subsubsection{Long variation estimates}\label{Sec5.1.1}
By the Calder\'{o}n-Zygmund decomposition mentioned above, we have
\begin{align*}
&|\{\lv(f_1,f_2)(x)>\al\}|  \\
&\quad\le  |\{\lv(g_1,g_2)(x)>\al/4\}| + |\{\lv(g_1, b_2)(x)>\al/4\}|\\
&\qquad+|\{\lv(b_1,g_2)(x)>\al/4\}|+|\{\lv(b_1,b_2)(x)>\al/4\}|,
\end{align*}
where $\lv$ represents the long variation operator defined in Section \ref{Sec3}.
Next we estimate each of these terms separately.

\noindent{\it Estimate for the first term.} We note that $g_2\in  L_c^\infty$ and
$\|g_1\|_2^2\le \|g_1\|_\infty \|g_1\|_1\lesssim \al^{p}$ by property (viii)  of the Calder\'on-Zygmund decomposition.
Now the $L^2 \times L^\infty \rightarrow L^2$ boundedness of $\lv$ gives
\begin{align*}
|\{\lv(g_1,g_2)>\al/4\}|  & \lesssim \al^{-2} \|\lv(g_1,g_2)\|_2^2\\
& \lesssim \al^{-2} \|g_1\|_2^2 \|g_2\|_\infty^2 \lesssim  \al^{-2} \cdot \al^{p} \cdot (\al^{p/p_2})^2=\al^{-p}.
\end{align*}


\noindent{\it Estimate for the second term.} By  the trivial pointwise estimate
\begin{equation*}
  \lv(g_1,b_2)(x)\le 2\Big(\sum_{k\in \Z} |\A_{2^k}^G(g_1,b_2)(x)|^2 \Big)^{1/2},
\end{equation*} it suffices to show
$$|\{x\in \R^d: (\sum_k |\A_{2^k}^G(g_1, b_2)(x)|^2)^{1/2} >\al/8\}| \lesssim \al^{-p}. $$

We use $Q_{i,j}$ (or $Q_{i,j'}$) to denote the cubes stemming from the Calder\'{o}n-Zygmund decomposition,
and $Q'_{i,l}$ to denote the cubes in $\mathscr{D}_{n}.$
Denote by $Q^*$ a certain dimensional dilate of a cube $Q$ in $\R^d$, and let
 $$\Omega=\bigcup_{j,\, j'} (Q_{1,j}\cup Q_{2,j'}),\quad\Omega^*=\bigcup_{j,\, j'} (Q_{1,j}^*\cup Q_{2,j'}^*).$$
By property (vi) of the Calder\'{o}n-Zygmund decomposition,
\begin{equation}\label{comple}
|\Omega^*|\le \sum_{j,\, j'} (|Q_{1,j}^*|+|Q_{2,j'}^*|) \lesssim \sum_{j,\,j'} (|Q_{1,j}|+|Q_{2,j'}|)\lesssim \al^{-p}.
\end{equation}
It thus suffices to show
\begin{equation}\label{new5.2}
|\{x\in \R^d\backslash \Omega^*: (\sum_k |\A_{2^k}^G(g_1, b_2)(x)|^2)^{1/2} >\al/8\}|\lesssim \al^{-p}.
\end{equation}
By Chebyshev's inequality, we have
\begin{align*}
& |\{x\in \R^d\backslash \Omega^*: (\sum_k |\A_{2^k}^G(g_1, b_2)(x)|^2)^{1/2} >\al/8\}|
 \lesssim \al^{-2} \int_{\R^d \backslash \Omega^*} \sum_k |\A_{2^k}^G(g_1, b_2)(x)|^2\, dx.
\end{align*}

For each $n\in \Z$ and $i=1,2$, let $\Q_n^{i}$ be the collection of cubes $Q_{i,j}$ of side length $2^n.$ Define
$$h_{i,n}=\sum_{j: \,Q_{i,j}\in \Q_n^i} b_{i,j}, \quad i=1,2.$$
Then
$$
b_i=\sum_{n\in \Z}\, \sum_{j:\, Q_{i,j} \in \Q_n^i} b_{i,j} =\sum_{n\in \Z} h_{i,n}.
$$

Notice that $\A_{2^k}^G(g_1, h_{2,n})(x)=0$ if $n\ge k$ since $(x,x)+ G_{2^k}$  is disjoint from the strip $\R^d\times Q_{2,j}$ for any $Q_{2,j}$ contained in the support of $h_{2,n}.$
It thus suffices to show
\begin{equation} \label{eq6.1}
 \int_{\R^d \backslash \Omega^*} \sum_k |\A_{2^k}^G(g_1, \sum_{n<k} h_{2,n})(x)|^2\, dx \lesssim \al^{2-p}.
\end{equation}
Let $\1_{i,n}=\sum_{Q_{i,j}\in \Q_n^i} \, \1_{Q_{i,j}}, i=1,2$.
Then \eqref{eq6.1} will follow if we show that
\begin{equation}\label{6.2}
|\A_{2^k}^G(g_1,  h_{2,n})(x)|^2\lesssim 2^{n-k} \cdot \al^{2-p} \A_{c2^{k}}^G(|g_1|, |\1_{2,n}|)(x), \qquad \forall x\in \R^d, n<k,
\end{equation}
where $c=\sqrt{d}+1.$
Indeed, integrating both sides of \eqref{6.2}, applying Minkowski's inequality and properties (vi) and (viii) of Calder\'on-Zygmund decomposition, we get
\begin{align}\label{5.10}
\begin{split}
\|\A_{2^k}^G(g_1,  h_{2,n})\|_{2}^2 &\lesssim 2^{n-k} \cdot \al^{2-p} \|\A_{c2^{k}}^G(|g_1|, |\1_{2,n}|)\|_1\\
 & \le 2^{n-k} \cdot \al^{2-p} \frac{1}{|G_{c2^k}|}\int_{G_{c2^{k}}}   \|g_1(\cdot+y_1)\1_{2,n}(\cdot+y_2)\|_{1}\, dy_1dy_2 \\
 & {\le 2^{n-k} \cdot \al^{2-p}   \|g_1\|_{\infty} \|\1_{2,n}\|_{1}.}
\end{split}
\end{align}
{Applying Lemma \ref{almost}, (ii) with $\S_k=\A_{2^k}^G,$ $\sigma(j)=2^{-\frac{|j|}{2}}$, $u_n=g_1$, $v_n=h_{2,n}$ and $a(u_n,v_n)=(\al^{2-p}   \|u_n\|_{\infty} \|v_{n}\|_{1})^{1/2}$}, the fact that the cubes in $\Q_n^2$ are pairwise disjoint,
and property (vi) of the Calder\'{o}n-Zygmund decomposition, we conclude that
\begin{align*}
{\text{LHS of }\eqref{eq6.1}\le \al^{2-p}   \|g_1\|_{\infty} \sum_{n\in \Z}\|\1_{2,n}\|_{1}
\lesssim  \al^{2}   \sum_{n\in \Z}\, \sum_{j:\,Q_{2,j}\in \Q_n^2} |Q_{2,j}| \lesssim \al^{2-p}.}
\end{align*}

To finish the estimate for the second term, it remains to show \eqref{6.2}.
Assume $n<k.$ It follows from property (v) of Calder\'on-Zygmund decomposition and H\"older's inequality that
\begin{equation}\label{new5.4}
\|b_{2,j}\|_{L^1(Q_{2,j})}\le \|b_{2,j}\|_{L^{p_2}} |Q_{2,j}|^{1-\frac{1}{p_2}} \lesssim \al^{\frac{p}{p_2}}|Q_{2,j}|.
\end{equation}
Write the support of $g_1\otimes h_{2,n}$ as
$$
\R^d\times \Big(\bigcup_{j:Q_{2,j}\in \Q_n^2} Q_{2,j}\Big)
=\bigcup_{Q_{1,l}'\times Q_{2,j} \in \mathscr{D}_{n}\times \Q_n^2} Q_{1,l}'\times Q_{2,j}
=:\bigcup_{\widetilde{Q}_{l,j}\in \mathscr{D}_{n}\times \Q_n^2} \widetilde Q_{l,j},
$$
where $\widetilde Q_{l,j}=Q_{1,l}'\times Q_{2,j},$
which leads to the decomposition
\begin{align*}
&\int_{(x,x)+ G_{2^k}} g_1(y_1) h_{2,n}(y_2)\, dy_1dy_2\\
&\qquad =\sum_{\widetilde Q_{l,j} \in \mathscr{D}_{n}\times \Q_n^2 }\int_{[(x,x)+ G_{2^k}]\, \cap\, \widetilde Q_{l,j}} g_1(y_1) b_{2,j}(y_2)\, dy_1dy_2.
\end{align*}
Since each $b_{2,j}$ has mean zero, the integral on the right hand side is nonzero only if $\widetilde Q_{l,j}$ intersects with the boundary of the ball $(x,x)+ G_{2^k}.$ Denote by $\I(G_{2^{k}})$ the collection of such cubes $\widetilde Q_{l,j}$, i.e.,
 $$\I(\partial G_{2^{k}})=\{\widetilde Q_{l,j}\in \mathscr{D}_{n}\times \Q_n^2:\, [(x,x)+ \partial G_{2^k}]\cap \widetilde Q_{l,j}\not= \varnothing\}.$$
For $n<k,$ the union of the cubes $\widetilde Q_{l,j}\in \I(\partial G_{2^{k}})$
is contained in $(x,x)+ G_{c2^{k}}$ for  $c=\sqrt{d}+1$. Hence
\begin{align}\label{new6.4}
\begin{split}
&\left|\int_{(x,x)+ G_{2^k}} g_1(y_1) h_{2,n}(y_2)\, dy_1dy_2\right|\\
&\le \sum_{ \widetilde Q_{l,j} \in \I(\partial G_{2^{k}})}
\left(\int_{Q_{1,l}'} |g_1(y_1)|\, dy_1\right)\cdot \left(\int_{Q_{2,j}} |b_{2,j}(y_2)|\, dy_2\right)\\
&\lesssim  \sum_{ \widetilde Q_{l,j} \in \I(\partial G_{2^{k}})}
\left(\int_{Q_{1,l}'} |g_1(y_1)|\, dy_1\right)\cdot (\al^{\frac{p}{p_2}} |Q_{2,j}|)\\
&= \al^{\frac{p}{p_2}} \sum_{ \widetilde Q_{l,j} \in \I(\partial G_{2^{k}})}
\int_{\widetilde Q_{l,j}} |g_1(y_1)|\1_{Q_{2,j}}(y_2)\, dy_1dy_2\\
&\le \al^{\frac{p}{p_2}}  \int_{(x,x)+G_{c2^{k}}} |g_1(y_1)|\1_{2,n}(y_2) \, dy_1dy_2,
\end{split}
\end{align}
where in the second inequality we used \eqref{new5.4}. Therefore,
\begin{equation}\label{new5.6}
  |\A_{2^k}^G( g_1, h_{2,n})(x)|\lesssim \al^{\frac{p}{p_2}}   \A_{c2^{k}}^G( |g_1|, \1_{2,n})(x).
\end{equation}

On the other hand, using $\|g_1\|_\infty \lesssim \al^{p}$ we have
\begin{align}\label{new5.7}
\begin{split}
&\left|\A_{2^k}^G( g_1, h_{2,n})(x)\right|\\
&\le \frac{1}{|G_{2^k}|} \sum_{ \widetilde Q_{l,j} \in \I(\partial G_{2^{k}})}
\left(\int_{Q_{1,l}'} |g_1(y_1)|dy_1\right)\cdot \left(\int_{Q_{2,j}} |h_{2,n}(y_2)|dy_2\right)\\
&\lesssim \al\cdot \frac{1}{|G_{2^k}|} \sum_{ \widetilde Q_{l,j} \in \I(\partial G_{2^{k}})} |\widetilde Q_j|\lesssim \al\cdot \frac{2^n 2^{(2d-1)k} }{2^{2kd}}=2^{n-k} \, \al.
\end{split}
\end{align}
Combining \eqref{new5.6} with \eqref{new5.7}, we obtain \eqref{6.2}. This completes all estimates for the second term.

\noindent{\it Estimate for the third term}.
The estimate of the third term is similar to that of the second one, the details being omitted.

\noindent{\it Estimate for the last term}. We write
$$b_1=\sum_{m\in \Z} \sum_{Q_{1,i}\in \Q_m^1} b_{1,i}=:\sum_{m\in \Z} h_{1,m},$$
and similarly for $b_2.$ Define
$$\widetilde h_{i,n}=\sum_{m\le n} h_{i,m}, \qquad \widetilde{\1}_{i,n}=\sum_{m\le n} \1_{i,m},\quad i=1,2.$$
As above, 
$\A_{2^k}^G(h_{1,m},h_{2,n})(x)=0$ if $k\le m$ or $k\le n$
since $(x,x)+ G_{2^k}$ is disjoint from any of $Q_{1,j}\times Q_{2,j'}$ in the support of $h_{1,m} \otimes h_{2,n}.$
These discussions together with Chebyshev's inequality yields
\begin{align*}
& |\{x\in \R^d \backslash \Omega^*: V_q(\A_{2^k}^G(b_1,b_2)(x):\, k\in \Z)>\al/8\}|\\
 &\lesssim \al^2 \int_{\R^d \backslash \Omega^*} \sum_k |\A_{2^k}^G(\sum_{m<k} h_{1,m}, \sum_{n<k} h_{2,n})(x)|^2\, dx\\
  &\lesssim \al^2 \int_{\R^d \backslash \Omega^*} \sum_k  |\sum_{n<k}\A_{2^k}^G(\sum_{m\le n} h_{1,m},  h_{2,n})(x) |^2\, dx\\
   &\qquad+ \al^2 \int_{\R^d \backslash \Omega^*} \sum_k  |\sum_{m<k}\A_{2^k}^G( h_{1,m}, \sum_{n<m} h_{2,n})(x) |^2\, dx\\
&= \al^2 \int_{\R^d \backslash \Omega^*} \sum_k  |\sum_{n<k}\A_{2^k}^G(\widetilde h_{1,n},  h_{2,n})(x) |^2\, dx\\
   &\qquad+ \al^2 \int_{\R^d \backslash \Omega^*} \sum_k  |\sum_{m<k}\A_{2^k}^G( h_{1,m}, \widetilde h_{2,m-1})(x) |^2\, dx.
\end{align*}
Since these two terms can be handled in the same way, we only treat the first. We need to show
 \begin{equation} \label{eq6.4}
\sum_{k\in \Z} \int_{\R^d \backslash \Omega^*}  |\sum_{n<k} \A_{2^k}^G(\widetilde h_{1,n},  h_{2,n})(x)|^2\, dx \lesssim \al^{2-p}.
\end{equation}

We claim that there exists a constant $c>1$ such that
\begin{equation}\label{claim5.11}
|\A_{2^k}^G(\widetilde h_{1,n},  h_{2,n})(x)|^2\lesssim 2^{n-k} \cdot \al^{2} \A_{c2^{k}}^G(\widetilde \1_{1,n}, \1_{2,n})(x).
\end{equation}
Assume the claim for the moment. Let us prove \eqref{eq6.4}.
Integrating both sides and using the trivial estimate $\widetilde{\1}_{1,n}\le 1$ yield
\begin{align}
\begin{split}
\|\A_{2^k}^G(\widetilde h_{1,n},  h_{2,n})\|_{2}^2 & \lesssim 2^{n-k} \cdot \al^{2} \|\A_{c2^{k}}^G(|\widetilde d_{1,n}|, |d_{2,n}|)\|_1\\
& \le 2^{n-k} \cdot \al^{2} \frac{1}{|G_{c2^{k}}|}\int_{G_{c2^{k}}}
 \sum_{Q_{2,j}\in \Q_n^2} |Q_{2,j}| \, dy_1dy_2 \\
&=  2^{n-k} \cdot \al^{2}\sum_{Q_{2,j}\in \Q_n^2} |Q_{2,j}|.
\end{split}
\end{align}
Then \eqref{eq6.4} follows by utilizing Lemma \ref{almost}, (i) with {$\S_k=\A_{2^k}^G$, $\sigma(j)=2^{-\frac{|j|}{2}}$ and $a(u_n,v_n)=\al (\sum_{Q_{2,j}\in \Q_n^2} |Q_{2,j}|)^{1/2}$}.

Finally, let us prove \eqref{claim5.11}. Noting that $n<k$ and using property (vii) of Calder\'on-Zygmund decomposition $\|b_{1,j}\|_{L^1(Q_{1,j})} \lesssim \al^{p}|Q_{1,j}|$ and $\|b_{2,j}\|_{L^1(Q_{2,j})} \lesssim \al^{\frac{p}{p_2}}|Q_{2,j}|$, we have
\begin{align*}
&\Big|\int_{(x,x)+ G_{2^k}} \widetilde h_{1,n}(y_1) h_{2,n}(y_2)\,dy_1dy_2\Big|\\
&\le \sum_{ \widetilde Q_{l,j} \in \I(\partial G_{2^{k}})}   \int_{\wt{Q}_{l,j}} |\widetilde h_{1,n}(y_1) b_{2,j}(y_2)|\,dy_1dy_2\\
&\le \sum_{ \widetilde Q_{l,j} \in \I(\partial G_{2^{k}})}
\left(\int_{Q_{1,l}'} |\widetilde h_{1,n}(y_1)|\,dy_1\right)\cdot \left(\int_{Q_{2,j}} |b_{2,j}(y_2)|\,dy_2\right)\\
&\lesssim \al \sum_{ \widetilde Q_{l,j} \in \I(\partial G_{2^{k}})}
 \bigg| Q_{1,l}'\cap \Big(\bigcup_{m\le n}\bigcup_{Q_{1,j'}\in \Q_m^1}Q_{1,j'}\Big)\bigg|\cdot |Q_{2,j}|\\
&= \al \sum_{ \widetilde Q_{l,j} \in \I(\partial G_{2^{k}})}
\int_{\widetilde Q_{l,j}} |\widetilde \1_{1,n}(y_1) \1_{2,n}(y_2)|\, dy_1dy_2\\
&\le \al  \int_{(x,x)+G_{c2^{k}}} |\widetilde \1_{1,n}(y_1) \1_{2,n}(y_2)|\, dy_1dy_2,
\end{align*}
where $\widetilde Q_{l,j}=Q_{1,l}'\times Q_{2,j}$.

We also have
\begin{align*}
&\frac{1}{|G_{2^k}|}\left|\int_{(x,x)+ G_{2^k}} \widetilde h_{1,n}(y_1) h_{2,n}(y_2)dy_1dy_2\right|\\
&\lesssim  \frac{\al}{|G_{2^k}|} \sum_{ \widetilde Q_{l,j} \in \I(\partial G_{2^{k}})}
 \bigg| Q_{1,l}'\cap \Big(\bigcup_{m\le n}\bigcup_{Q_{1,j'}\in \Q_m^1}Q_{1,j'}\Big)\bigg|\cdot |Q_{2,j}|\\
&\le \frac{\al}{|G_{2^k}|} \sum_{ \widetilde Q_{l,j} \in \I(\partial G_{2^{k}})}
|\widetilde Q_{l,j}|  \lesssim \al \cdot \frac{2^n 2^{(2d-1)k} }{2^{2kd}}=2^{n-k} \, \al.
\end{align*}
Combining the above two estimates yields \eqref{claim5.11}, and hence concluding the proof of $L^1\times L^{p_2}\rightarrow L^{p,\infty}$ long variation estimate.

\subsubsection{Short variation estimates}
By the triangle inequality,
\begin{align*}
& \Big|\Big\{\svq(f_1,f_2)>\al\Big\}\Big| \\
  &\quad \le \Big|\Big\{\svq(g_1,g_2) >\al/4\Big\}\Big|  + \Big|\Big\{\svq(g_1, b_2)>\al/4\Big\}\Big| \\
  &\qquad + \Big|\Big\{\svq(b_1,g_2)>\al/4\Big\}\Big|  + \Big|\Big\{\svq(b_1,b_2)>\al/4\Big\}\Big|.
\end{align*}
The first term can be treated using the $L^2\times L^\infty \rightarrow L^2$ estimate as in the long variation case.

For the second term, by an argument similar to that used in the proof of long variation estimates {for the second term}, we only need to show
\begin{equation}\label{5.12}
\sum_{k\in \Z} \int_{\R^d \backslash \Omega^*}  \sum_{I_i\in \ss_k} |\sum_{n<k} [\A_{t_{i+1}}^G(g_1,  h_{2,n})-\A_{t_{i}}^G(g_1, h_{2,n})]|^2 \lesssim \al^{2-p}.
\end{equation}
{Recalling that the bi-subadditive operator ${\g}_k$ was defined by
$$\g_k(f_1,f_2)(x)=\Big(\sum_{I_i\in \ss_k} |\A_{t_{i+1}}^G(f_1,f_2)(x)-\A_{t_{i}}^G(f_1,f_2)(x)|^2 \Big)^{1/2},$$
by Minkowski's inequality for series, it suffices to show
\begin{equation}\label{}
\sum_{k\in \Z}   \|\sum_{n<k} \g_k(g_1,  h_{2,n})\|_{L^2(\R^d \backslash \Omega^*)}^2 \lesssim \al^{2-p}.
\end{equation}
Applying Lemma \ref{almost}, (i) and using the simple estimate
$$\|\A_{c2^{k+1}}^G(|g_1|, \1_{2,n})\|_1 \lesssim \|g_1\|_\infty \|\1_{2,n}\|_1,$$
matters are reduced to showing the pointwise estimate
\begin{equation}\label{6.6}
\g_k(g_1,  h_{2,n})(x)^2 \lesssim \al^{2-p} 2^{n-k} \A_{c2^{k+1}}^G(|g_1|, \1_{2,n})(x),\qquad \forall~ x\in \R^d.
\end{equation}
}
Since the left hand side is majorized by a constant multiple of $J_1+J_2$, where
\begin{align*}
J_1 & = \sum_{I_i\in \ss_k} \bigg|\frac{1}{|G_{t_{i+1}}|} \int_{(x,x)+(G_{t_{i+1}}\backslash G_{t_i})} (g_1\otimes  h_{2,n})\bigg|^2\\
\intertext{and}
J_2 & =\sum_{I_i\in \ss_k}\left (\frac{1}{|G_{t_i}|} -\frac{1}{|G_{t_{i+1}}|} \right)^2  \bigg|\int_{(x,x)+ G_{t_i}} (g_1\otimes  h_{2,n})\bigg|^2,
\end{align*}
it suffices to bound $J_1$ and $J_2$ by the right hand side of \eqref{6.6}.

We treat first $J_2$ which is easier.
Denote
$$
\I(\partial G_{t_i})= \{\widetilde Q_{l,j}=Q_{1,l}'\times Q_{2,j}\in \mathscr{D}_{n}\times\Q_n^2: [(x,x)+\partial G_{t_i}]\cap \widetilde Q_{l,j}\not=\varnothing\}.
$$
By similar arguments to that in \eqref{new6.4}, we have
\begin{align*}
\sqrt{J_{2}}&\lesssim  \bigg( \sum_{I_i\in \ss_k}(\frac{1}{|G_{t_i}|} -\frac{1}{|G_{t_{i+1}}|} )\bigg
)   \sup_{I_i\in \ss_k}\left|\int_{(x,x)+ G_{t_i}} (g_1\otimes  h_{2,n})\right|\\
&\lesssim   \frac{1}{2^{2kd}} \sup_{I_i\in \ss_k} \sum_{ \widetilde Q_{l,j}\in \I(\partial G_{t_i}) }
\left(\int_{Q_{1,l}'} |g_1(y_1)|\, dy_1\right)\cdot \left(\int_{Q_{2,j}} |b_{2,j}(y_2)|\, dy_2\right) \\
&\lesssim   \frac{1}{2^{2kd}}  \sup_{I_i\in \ss_k} \sum_{ \widetilde Q_{l,j}\in \I(\partial G_{t_i}) }
\left(\int_{Q_{1,l}'} |g_1(y_1)|\, dy_1\right) \cdot (\al^{\frac{p}{p_2}} |Q_{2,j}|) \\
&\lesssim  \frac{1}{2^{2kd}} \al^{\frac{p}{p_2}}  \sup_{I_i\in \ss_k} \sum_{ \widetilde Q_{l,j}\in \I(\partial G_{t_i}) } \int_{\widetilde Q_{l,j}} |g_1(y_1) \1_{2,n}(y_2)|\, dy_1 dy_2 \\
&\lesssim \al^{\frac{p}{p_2}} \cdot  \A_{c2^{k+1}}^G (|g_1|,\1_{2,n})(x),
\end{align*}
and
\begin{align*}
\sqrt{J_{2}} &\lesssim   \frac{1}{2^{2kd}} \sup_{I_i\in \ss_k} \sum_{ \widetilde Q_{l,j}\in \I(\partial G_{t_i}) }
\Big(\int_{Q_{1,l}'} |g_1(y_1)|\, dy_1\Big)\cdot \Big(\int_{Q_{2,j}} |b_{2,j}(y_2)|\, dy_2\Big) \\
&\lesssim   \frac{1}{2^{2kd}}   \sup_{I_i\in \ss_k} \sum_{ \widetilde Q_{l,j}\in \I(\partial G_{t_i}) }
\left(\al^p |Q_{1,l}'|\right) \cdot (\al^{\frac{p}{p_2}} |Q_{2,j}|) \\
&=\al \cdot  \frac{1}{2^{2kd}}  \sup_{I_i\in \ss_k}  \sum_{\widetilde Q_{l,j}\in \I(\partial G_{t_i})} |\widetilde Q_{l,j}| \lesssim  \al \cdot  \frac{2^{k(2d-1)}\cdot 2^n}{2^{2kd}}= \al 2^{n-k}.
\end{align*}
Combining the above two estimates, we obtain \eqref{6.6}.

For $J_1$, we have
\begin{align*}
&\sup_{I_i\in \ss_k}\bigg|\frac{1}{|G_{t_{i+1}}|} \int_{(x,x)+(G_{t_{i+1}}\backslash G_{t_i})} g_1(y_1) h_{2,n}(y_2)\, dy_1dy_2\bigg|\\
&\le \frac{1}{2^{2kd}} \sup_{I_i\in \ss_k} \sum_{\widetilde Q_{l,j} \in \I(\partial(G_{t_{i+1}}\backslash G_{t_i})) }
\bigg(\int_{Q_{1,l}} |g_1(y_1)|\, dy_1\bigg)\cdot \bigg(\int_{Q_{2,j}} |h_{2,n}(y_2)|\, dy_2 \bigg)\\
&\lesssim  \frac{1}{2^{2kd}} \sup_{I_i\in \ss_k} \sum_{\widetilde Q_{l,j} \in \I(\partial(G_{t_{i+1}}\backslash G_{t_i}))} \left(\al^p |Q_{1,l}| \right)\cdot (\al^{\frac{p}{p_2}} |Q_{2,j}|)\\
&\lesssim \al\cdot \frac{2^n 2^{(2d-1)k}}{2^{2kd}}=\al\cdot 2^{n-k}.
\end{align*}

On the other hand,
\begin{align*}
&\sum_{I_i\in \ss_k} \left|\frac{1}{|G_{t_{i+1}}|}\int_{(x,x)+(G_{t_{i+1}}\backslash G_{t_i})} g_1(y_1) h_{2,n}(y_2)\, dy_1dy_2\right|\\
&\lesssim  \frac{1}{2^{2kd}}   \sum_{I_i\in \ss_k} \int_{(x,x)+(G_{t_{i+1}}\backslash G_{t_i})} |g_1(y_1) h_{2,n}(y_2)|\, dy_1dy_2\\
&\le \frac{1}{2^{2kd}}   \int_{(x,x)+G_{2^{k+1}}} |g_1(y_1) h_{2,n}(y_2)|\, dy_1dy_2\\
&\le \frac{1}{2^{2kd}}\sum_{ \widetilde Q_{l,l} \in \I(G_{2^{k+1}}) }
\left(\int_{Q_{1,j}'} |g_1(y_1)|\, dy_1\right)\cdot \left(\int_{Q_{2,j}} |h_{2,n}(y_2)|\, dy_2\right)\\
&\lesssim \al^{\frac{p}{p_2}} \frac{1}{2^{2kd}} \sum_{ \widetilde Q_{l,j} \in \I(G_{2^{k+1}}) }
\left(\int_{Q_{1,j}'} |g_1(y_1)|\, dy_1\right)\cdot |Q_{2,j}|\\
&= \al^{\frac{p}{p_2}} \frac{1}{2^{2kd}}\sum_{ \widetilde Q_{l,j} \in \I(G_{2^{k+1}}) }
\int_{\widetilde Q_{l,j}} |g_1(y_1) \1_{2,n}(y_2)|\, dy_1dy_2\\
&\le \al^{\frac{p}{p_2}}  \frac{1}{2^{2kd}} \int_{ G_{c2^{k+1}}} |g_1(y_1)\1_{2,n}(y_2)|\, dy_1dy_2\\
& \approx \al^{\frac{p}{p_2}}  \A_{c2^{k+1}}^G(|g_1|, \1_{2,n})(x).
\end{align*}
Combining these two inequalities, we obtain the desired bound for $J_1.$
This completes the estimate of the second term. The other terms can be similarly treated, the details being omitted.


\subsection{$L^1\times L^\infty \rightarrow L^{1,\infty}$ estimate}\label{Sec5.3}
Since the proof is similar to that in the case $1<p_2<\infty$,
we only sketch the proof of the long variation estimate.
Assume that $f_1\in L^1(\R^d)$ and $f_2\in L^\infty_c(\R^d)$ with $\|f_1\|_1=\|f_2\|_\infty=1.$
We need to show
$$|\{\lv(f_1,f_2)>\al\}| \lesssim \al^{-1}.$$
We adopt the same notations as used in previous subsection.
By the Calder\'{o}n-Zygmund decomposition $f_1=g_1+b_1$, we have
\begin{align*}
|\{\lv(f_1,f_2)>\al\}|  \le & |\{\lv(g_1,f_2)>\al/2\}| + |\{\lv(b_1, f_2)>\al/2\}|.
\end{align*}

 By the $L^2 \times L^\infty \rightarrow L^2$ boundedness of $\lv$
and the estimate $\|g_1\|_2^2\le \|g_1\|_\infty \|g_1\|_1\lesssim \al$, we have
\begin{align*}
|\{\lv(g_1,f_2)>\al/2\}|  & \lesssim \al^{-2} \|\lv(g_1,f_2)\|_2^2\\
& \lesssim \al^{-2} \|g_1\|_2^2 \|f_2\|_\infty^2 \lesssim  \al^{-2} \cdot \al=\al^{-1}.
\end{align*}

For the second term, by an argument similar to that used in Section \ref{Sec5.1.1}, we can reduce the matters to showing
\begin{equation} \label{eq7.1}
\sum_{k\in \Z} \int_{\R^d \backslash \Omega^*}  |\A_{2^k}^G(\sum_{n<k} h_{1,n}, f_2)(x)|^2\, dx \lesssim \al.
\end{equation}
Assuming $n<k,$ we claim that
\begin{equation}\label{7.2}
|\A_{2^k}^G(h_{1,n}, f_2)(x)|^2\lesssim 2^{n-k} \cdot \al^{2} \A_{c2^{k}}^G(\1_{1,n}, |f_2|)(x),\quad \forall \, x\in \R^d.
\end{equation}
Assume this claim holds for the moment, let us show \eqref{eq7.1}.
Integrating both sides and using the hypothesis that $\|f_2\|_\infty=1$, we get
\begin{align*}
\|\A_{2^k}^G( h_{1,n}, f_2)\|_{2}^2  & \lesssim 2^{n-k} \cdot \al^{2} \|\A_{c2^{k}}^G(\1_{1,n}, |f_2|)\|_1 \lesssim 2^{n-k} \cdot \al^{2}   \sum_{j: \, Q_{1,j}\in \Q_n^1} | Q_{1,j}|.
\end{align*}
Then \eqref{eq7.1} follows by applying Lemma \ref{almost} with $\S_k=\A_{2^k}^G$, $\sigma(j)=2^{-\frac{|j|}{2}}$, and {$a(h_{1,n}, f_2)= \al  ( \sum_{j: \, Q_{1,j}\in \Q_n^1} | Q_{1,j}|)^{1/2}$.}

It suffices to establish claim \eqref{7.2}. By  arguments similar to those given in \eqref{new6.4} and \eqref{new5.7} (with the good function $g_2$ replaced by the function $f_2$ itself in the current case), one can derive that
\begin{align*}
\frac{1}{|G_{2^k}|} \left|\int_{(x,x)+ G_{2^k}} h_{1,n}(y_1) f_2(y_2) \, dy_1dy_2\right| & \lesssim \frac{ \al }{|G_{2^k}|} \int_{(x,x)+G_{c2^{k}}} \1_{1,n}(y_1) |f_2(y_2)| \, dy_1dy_2 \\
\intertext{and}
 \frac{1}{|G_{2^k}|} \left|\int_{(x,x)+ G_{2^k}} h_{1,n}(y_1) f_2(y_2)\, dy_1dy_2\right| & \lesssim  2^{n-k} \cdot \al.
\end{align*}
Combining these two estimates yields \eqref{7.2}, and hence concludes the proof of long variation inequality.

The short variation estimate can be proved analogously, the details being omitted.
\section{Restricted weak type estimates and concluding the proof of Theorem \ref{thm:main1}}\label{Sec6}
In this section, we will conclude our proof of Theorem \ref{thm:main1} by multilinear real interpolation of Grafakos-Kalton \cite{GK} based on the previous estimates. To begin with, let us recall some definitions.

Denote by $S(\R^d)$ the space of functions of the form $\sum_{i=1}^N c_i \1_{E_i},$
where each measurable subset $E_i$ of $X$ has finite measure.
An operator $T$ defined on $S(\R^d)\times \cdots \times S(\R^d)$ and taking values in the set of complex-valued measurable functions on
$\R^d$ is called multisublinear if for all $1\le j\le m,$ all $f_j,g_j$ in $S(\R^d)$ and all $\la\in \Bbb C$ the statements
$$|T(f_1,\ldots, \la f_j,\ldots, f_m)|=|\la|\, |T(f_1,\ldots, f_j, \ldots, f_m)|, $$
$$|T(\ldots, f_j+g_j,\ldots)|\le |T(\ldots,f_j,\ldots)|+|T(\ldots,g_j,\ldots)|$$
hold almost everywhere.

\begin{defn}
Let $0<p_1,p_2,\cdots, p_m, p\le \infty.$ We say that an $m$-sublinear operator $T$ is of \textit{restricted weak type $(p_1,\ldots, p_m,p)$}
if there is a constant $C=C(p_1,\ldots, p_m, p)$ such that for all measurable subsets $A_1,\ldots, A_m$ of finite measure we have
$$\sup_{\la>0}\la |\{x\in \R^d:\, |T(\1_{A_1},\ldots, \1_{A_m})|>\la\}|^{\frac{1}{p}} \le C|A_1|^{\frac{1}{p_1}}\cdots |A_m|^{\frac{1}{p_m}}$$
when $p<\infty$ and
$$\|T(\1_{A_1},\ldots, \1_{A_m})\|_\infty \le C|A_1|^{\frac{1}{p_1}}\cdots |A_m|^{\frac{1}{p_m}}$$
when $p=\infty.$
\end{defn}

\noindent\textit{Proof of Theorem \ref{thm:main1}.} In order to prove the main result by using multilinear real interpolation (cf. \cite{GK}), we need to show that
$V_q (\A_{t}^G:\, t>0)$ is of restricted weak types $(1,\infty,1)$, $(\infty, 1, 1)$,  $(1,1,1/2),$  $(\infty, s, s)$ and $(s,\infty,  s)$, for any $2<s<\infty$.
The first three follows from the corresponding estimates obtained in previous sections by taking $f_1=\1_A$ and $f_2=\1_B$. It remains to show the
restricted weak types $(\infty, s, s)$ and $(s,\infty, s)$ for any $2<s<\infty$.

 For any fixed $f_1\in L^\infty\cap L^2,$ define a sublinear operator $T_{f_1}$ by
$$T_{f_1}(g)(x)=V_q(\A_t^G(f_1,g):\, t>0)(x), \quad \forall \, g\in L^{p}.
$$
From the boundedness results of $V_q(\A_t^G:\, t>0)$ proved in previous sections,  it follows that $T_{f_1}$ is bounded from $L^2$ to $L^2$, from $L^\infty_c$ to $BMO_\d$, and from $L^1$ to $L^{1,\infty}$, and the operator norms
are both bounded by a constant multiple of $\|f_1\|_\infty.$
By classical Marcinkiewicz interpolation for sublinear operators, $T_{f_1}$ is bounded on $L^s$ for all $1<s<\infty,$
with operator norm at most $C_{p,d}\|f_1\|_\infty.$ This implies that
$$\|V_q(\A_t^G(f_1,f_2):\, t>0)\|_{s}=\|T_{f_1}(f_2)\|_{s} \lesssim \|f_1\|_\infty \,\|f_2\|_s$$
 for any $1<s<\infty$, any $f_1\in L^\infty\cap L^2, \, f_2\in L^s.$
 Applying this with $f_1=\1_A$ and $f_2=\1_B$, we see that $V_q(\A_t^G:\, t>0)$ is of restricted weak type $(\infty, s, s)$
 for any $1<s<\infty.$ By symmetry, $V_q(\A_t^G:\, t>0)$ is of restricted weak type $(s,\infty, s)$ for any $1<s<\infty.$

 Finally, for $1<p_1,p_2<\infty$ and $1/p=1/p_1+1/p_2$,
  $(p_1,p_2,p)$ lies in the convex hull of $(1,\infty,1)$, $(\infty, 1, 1)$,  $(1,1,1/2),$
 $(\infty, s, s)$ and $(s,\infty,  s)$ for sufficiently large constant $s$. The required $L^{p_1}\times L^{p_2}\rightarrow L^p$ bound then follows by applying the multilinear real interpolation
 in Lemma \ref{interpolation}. This concludes the proof of Theorem \ref{thm:main1}.
 $\hfill\Box$

\medskip

 We conclude this section by showing that $V_q(\A_t^G(f_1,f_2))$ is, in general, not of the restricted weak type $(\infty,\infty,\infty)$.
  This means that the $L^\infty \times L^\infty\rightarrow BMO_\d$ estimate proved in Section \ref{Sec4} cannot be strengthened to $L^\infty\times L^\infty \rightarrow L^\infty$ estimate even for $f_1,f_2$ being characteristic functions of measurable sets of finite measures. We do this in the case $G=\widetilde{B}(0,1):=\widetilde{B}$, the unit ball of $\R^{2d}.$
To this end, we construct two sequences $(E_n)_n$ and $(F_n)_n$ of measurable sets of finite measures in $\R^d$ such that
$$\|V_q(\A_t^{\widetilde{B}}:\, t>0)(\1_{E_n},\1_{F_n})\|_{\infty}\gtrsim n,\quad \forall\, n\in \N. $$

Take $\alpha>1$ large enough so that $|[(B_1)^c\times \R^d] \cap \B_\al|>\frac{4}{5}|\B_\al|. $
There exists $\epsilon_0>0$ such that $|[(B_1)^c \times \R^d] \cap [(x,x)+\B_\al]|>\frac{3}{4}|\B_\al|$ for all $x\in B_{\ep_0}$.
 For each $n\in \N,$ define
$$E_n=\bigcup_{i=1}^n B_{\al^{2n+1}}\backslash B_{\al^{2n}}\quad \mbox{and}\quad F_n=B_{\al^{2n+2}}.$$
Then for $1\le i\le 2n+1$ and any $x\in B_{\ep_0}$,
$$\A_{\al^i}^{\widetilde{B}}(\1_{E_n},\1_{F_n})(x)=\frac{|[(x,x)+\B_{\al^i}]\cap [E_n\times F_n]|}{|\B_{\al^i}|}\begin{dcases}
>\frac{3}{4},& i\ \mbox{odd};\\
<\frac{1}{4},& i\ \mbox{even}.
\end{dcases}$$
 Let $t_i=\al^i$ for $i\in \N.$ Then for $0\le i\le 2n$ and $x\in B_{\ep_0}$,
$$\left|\A_{t_{i+1}}^{\widetilde{B}}(\1_{E_n},\1_{F_n})-\A_{t_{i}}^{\widetilde{B}}(\1_{E_n},\1_{F_n})\right| >\frac{1}{2},$$
which implies
$$V_q(\A_t^{\widetilde{B}}(\1_{E_n},\1_{F_n}):\, t>0)(x)>2^{-q+1}n,\quad \forall \ x\in B_{\ep_0}.$$
This shows that $V_q(\A_t^{\widetilde{B}}:\, t>0)$ is not of restricted weak type $(\infty,\infty,\infty).$

\section{Bounds for new bilinear square functions involving conditional expectation}\label{Sec7}

Recall that the square function $\L$ is defined by
$$
\L(f_1,f_2)(x)=\bigg( \sum_{k\in \Z}|\L_k(f_1, f_2)(x)|^2 \bigg)^{\frac{1}{2}},
$$
where $$\L_k(f_1, f_2)(x)=\A_{2^k}^G(f_1,f_2)(x)-\E_k f_1 (x) \cdot \E_k f_2(x).$$

As a byproduct of our arguments, we obtain the $L^{p_1}\times L^{p_2} \to L^p$ boundedness of the bilinear square function $\L$,
which is the content of the following
\begin{thm}\label{square}
The square function $\L$ is bounded from
$L^{p_1} \times L^{p_2}$ to $L^p$ for all
Let $1<p_1,p_2<\infty$ and $1/p=1/p_1+1/p_2,$ bounded from $L^\infty_c \times L^\infty_c$ to $BMO_\d,$
and bounded from $L^1\times L^{p_2}$ to $L^{p,\infty}$ for all $1\leq p_2\le \infty$ and $1/p=1+1/p_2.$
\end{thm}
\begin{proof}
By arguments similar to that given in the proof of Theorem \ref{thm:main1}, it suffices to show the $L^2\times L^\infty\to L^2$, $L^\infty \times L^\infty \to BMO$,
and $L^1\times L^{p_2}\to L^{p,\infty}$ bounds of $\L$. The $L^2\times L^\infty\to L^2$ has been established in Section \ref{Sec3}.

Let us now show the $L^\infty \times L^\infty \to BMO$ bound. Following similar arguments as in Section \ref{Sec4}, we have
\begin{align*}
 \frac{1}{|Q|} \int_Q |\L(f_1,f_2)-a_Q|
 & \le   \frac{1}{|Q|} \int_Q |\L(f_1^0,f_2^0)| + \frac{1}{|Q|} \int_Q |\L(f_1^0,f_2^\infty)| \\
  & \qquad + \frac{1}{|Q|} \int_Q |\L(f_1^\infty,f_2^0)| + \frac{1}{|Q|} \int_Q |\L(f_1^\infty, f_2^\infty)-\L(f_1^2, f_2^\infty)(c_Q)| \\
 &  =: II_1+II_2+II_3+II_4.
\end{align*}

The required bounds for $II_1, II_2$ and $II_3$ follows easily from the $L^\infty \times L^2\rightarrow L^2$ boundedness of $\L$ as before.
The estimate for $II_4$ will follow if we show that,
for any $x\in Q,$
\begin{align}\label{e7.1}
\begin{split}
&\Big(\sum_{k\in \Z} |[\A_{2^k}^G (f_1^\infty,f_2^\infty)(x) - \E_k f_1^\infty(x) \cdot \E_k f_2^\infty(x)] \\
&\qquad -[\A_{2^k}^G (f_1^\infty,f_2^\infty)(c_Q)-\E_k f_1^\infty(c_Q) \cdot \E_k f_2^\infty(c_Q)] |^2 \Big)^{1/2}
 \lesssim \|f_1\|_\infty \|f_2\|_\infty,
\end{split}
\end{align}
where the implicit constant is independent of $x$ and $Q.$

We consider two cases as before. If $2^k<\ell(Q),$ then for any $x\in Q$,
both $\E_k f_1^\infty$ and $\E_k f_2^\infty$ are supported in $\R^d\backslash Q$,
so
$$\E_k f_1^\infty(x)\cdot \E_k f_2^\infty(x)=\E_k f_1^\infty(c_Q) \cdot \E_k f_2^\infty(c_Q)=0.$$

If $2^k\ge \ell(Q)$, then
in this case, there exists an atom in $\mathscr{D}_k$ that contains $Q,$
so $\E_k f_1^\infty(x) \cdot \E_k f_2^\infty(x)=\E_k f_1^\infty(c_Q) \cdot \E_k f_2^\infty(c_Q).$

In both cases, the conditional expectations on the left side of \eqref{e7.1} are equal to zero, and hence \eqref{e7.1} follows from \eqref{new4.1}. This completes the proof of the $L^\infty \times L^\infty \to BMO$ bound of $\L.$

Finally, let us show the $L^1\times L^{p_2}\to L^{p,\infty}$ bounds of $\L$.
We will see that again the conditional expectation will play no role and the estimates follows from
corresponding estimates in Section \ref{Sec4}.
By the Calder\'{o}n-Zygmund decomposition, we have
\begin{align*}
|\{\L(f_1,f_2)(x)>\al\}|
&\le  |\{\L(g_1,g_2)(x)>\al/4\}| + |\{\L(g_1, b_2)(x)>\al/4\}|\\
&\qquad+|\{\L(b_1,g_2)(x)>\al/4\}|+|\{\L(b_1,b_2)(x)>\al/4\}|,
\end{align*}
The first term can be treated as in Section \ref{Sec4} by using the $L^2\times L^\infty \to L^2$ bound of $\L$, together with the properties of good functions $g_1$ and $g_2$.

To estimate the second term, by \eqref{comple}, it suffices to prove
\begin{equation}\label{e7.2}
|\{x\in \R^d\backslash \Omega^*: (\sum_k |\A_{2^k}^G(g_1, b_2)(x)-\E_kg_1(x) \cdot \E_k b_2(x)|^2)^{1/2} >\al/8\}|\lesssim \al^{-p}.
\end{equation}
Write
$$
b_2=\sum_{n\in \Z}\, \sum_{j:\, Q_{2,j} \in \Q_n^2} b_{2,j} =\sum_{n\in \Z} h_{2,n}.
$$
We shall show that for $x\notin \Omega^*,$ $\E_k b_2(x)=0$ for every $k.$
Indeed, if $k\le n,$ then $\E_k h_{2,n}(x)=0$ since the atom of $\mathscr{D}_k$ containing $x$
is disjoint from the support of $h_{2,n};$ if $n<k,$ for each $Q_{2,j}\in \Q_n^2,$
there exists an atom of $\mathscr{D}_k$ containing $Q_{2,j}.$ Then $\E_k b_{2,j}=0$ due to the cancellation condition of $b_{2,j},$
and hence $\E_k h_{2,n}(x)=0$.
Now \eqref{e7.2} is equivalent to \eqref{new5.2} established in Section \ref{Sec5}.
The other terms can be treated in the same way.

This concludes the proof of the $L^1\times L^{p_2}\to L^{p,\infty}$ bounds of $\L$, and hence Theorem \ref{square}.
\end{proof}
\section{Applications}\label{Sec8}

\subsection{{Discrete bilinear averaging operators}}

Let $1<p_1,p_2<\infty$. Consider the discrete bilinear averaging operator defined by
$$\mathbf{A}_t^{G,\d}(f_1, f_2)(j)=\frac{1}{\sharp(G_t\,\cap\, \Z^{2d})} \sum_{(k,m)\in G_t\, \cap\, \Z^{2d}} f_1(j-k)\, f_2(j-m),\quad t> 0,\ j\in \Z^d$$
for $f_1\in \ell^{p_1}(\Z^d)$ and $f_2\in \ell^{p_2}(\Z^d)$,
where $\sharp(S)$ is the cardinality of the set $S.$

By arguments similar to that used in the proof of Theorem \ref{thm:main1}, we can prove
\begin{cor}
Let $2<q<\infty.$ Then for any $1<p_1,p_2<\infty,$ $1/2<p<\infty$ with $1/p=1/p_1+1/p_2$, there exists a constant $C_{p,q}$ such that
$$
\|V_q((\A_t^{G,\d})_{t> 0}) (f_1, f_2)\|_{\ell^p} \le C_{p,q} \|f_1\|_{\ell^{p_1}} \|f_2\|_{\ell^{p_2}}.
$$
\end{cor}
\begin{rem}
This result provides a bilinear analogue of the result obtained by Jones-Rosenblatt-Wierdl \cite{JRW2}.
\end{rem}
\subsection{Bilinear averaging operators of Demeter-Tao-Thiele}
 Demeter, Tao, and Thiele \cite{DTT} studied
a general class of multilinear averaging operators and proved maximal inequalities for them.
We consider a special case where the averaging operator takes the form
$$
M_{\La,t}(f_1,f_2)(x)=\frac{1}{t^{2d}}\int_{\substack{|u_1|< t\\ |u_2|< t}} f_1(x+\la_{1,1}u_1+\la_{1,2}u_2)f_2(x+\la_{2,1}u_1+\la_{2,2}u_2)\, du_1\, du_2.
$$
Here $\Lambda=\{\lambda_{i,j}\}_{2\times 2}$ is assumed to be nonsingular.

This operator can be expressed as an averaging operator over convex bodies. Indeed, let $\Ga$ be the inverse of $\Lambda$, which is also nonsingular.
Then the following set
$$G_\Ga=\{(y_1,y_2)\in \R^{2d}: |\ga_{1,1}y_1+\ga_{1,2}y_2|< 1,\, |\ga_{2,1}y_1+\ga_{2,2}y_2|< 1\}$$
is a symmetric convex body in $\R^{2d}$, whose ``volume'' is given by the absolute value of its determinant $|\det(\Ga)|$.
Changing the variable we have
\begin{align*}
M_{\La,t}(f_1,f_2)(x)&=\frac{1}{|\det(\Ga)|\cdot t^{2d}}\int_{(G_\Ga)_t} f_1(x+y_1)f_2(x+y_2)\, dy_1\, dy_2\\
&=\frac{1}{|(G_\Ga)_t|}\int_{(G_\Ga)_t} f_1(x+y_1)f_2(x+y_2)\, dy_1\, dy_2 = \A^{G_\Ga}_t(f_1,f_2)(x).
\end{align*}
Then the following corollary follows immediately from Theorem \ref{thm:main1}.

\begin{cor}

Let $2< q<\infty.$ If $\Lambda$ is nonsingular,  then for any $1<p_1,p_2<\infty,$ $1/2<p<\infty$ with $1/p=1/p_1+1/p_2$,  there exists a constant $C_{p,q}$ such that
$$
\|V_q(M_{\La,t})_{t> 0}) (f_1, f_2)\|_{L^p(\R^d)} \le C_{p,q} \|f_1\|_{L^{p_1}(\R^d)} \|f_2\|_{L^{p_2}(\R^d)}.
$$
\end{cor}

\subsection{Ergodic bilinear averaging operators}
Let $(\Omega, \mu)$ be a measurable space.
For $x\in \R^d,$ let $T^x: \Omega\rightarrow \Omega$ be a family of measure preserving transformations on $(\Om,\mu)$ which satisfy the following semigroup properties: $T^{x+y}=T^x\cdot T^y,$ $\forall \ x,y\in \R^{d}$.

For $t>0$, $f_1,f_2\in L^1_{loc}(\Omega),$ define a family of ergodic bilinear averaging operators by:
$$\mathcal{A}_t(f_1,f_2)(\w)=\frac{1}{|G_t|}\int_{G_t}  f_1(T^x\w)\cdot f_2(T^y\w)\, dx\, dy,\quad \w\in \Om.$$

\begin{cor}\label{cor7.2}
Let $2< q<\infty.$ Then for any $1<p_1,p_2<\infty,$ $1/2<p<\infty$ with $1/p=1/p_1+1/p_2$,  there exists a constant $C_{p,q}$ such that
$$
\|V_q((\mathcal{A}_t)_{t> 0}) (f_1, f_2)\|_{L^p(\Om)} \le C_{p,q} \|f_1\|_{L^{p_1}(\Om)} \|f_2\|_{L^{p_2}(\Om)}.
$$
\end{cor}

\begin{proof}
The proof uses the transference principle.
For a family $(f_t)_{t>0}$ of functions, we use the notation $\|f_t\|_{V_q(t\in (0,R))}:=V_q((f_t)_t:\, t\in (0,R)).$
For any fixed $R>0,$ we have
\begin{align}\label{8.1}
\begin{split}
&\left\|\Big\|(\frac{1}{|G_t|}\int_{G_t}f_1(T^x\om) f_2(T^y\om)\, dx\, dy)\Big\|_{V_q(t\in (0,R))}\right\|_{L^p(\Om)}^p\\
&=\left\|\frac{1}{|G_t|}\int_{G_t}f_1(T^zT^{x-z}\om) f_2(T^z T^{y-z}\om)\, dx\, dy\right\|_{L^p(\Om, V_q(t\in (0,R)))}^p\\
&=\frac{1}{c_dJ^d}\int_{|z|\le J}\Big\|\frac{1}{|G_t|}\int_{G_t}\widetilde{T}^z[f_1(T^{x-z}\cdot) f_2(T^{y-z}\cdot)](\om)\, dx\, dy \Big\|_{L^p(\Om,V_q(t\in (0,R)))}^p\, dz\\
&\lesssim\frac{1}{J^d}\int_{|z|\le J} \bigg\|(\widetilde{T}^z\otimes id)\Big(\frac{1}{|G_t|}\int_{G_t}[f_1(T^{x-z}\cdot)f_2(T^{y-z}\cdot)](\om)\, dx\, dy\Big)\bigg \|_{L^p(\Om,V_q(t\in (0,R)))}^p\, dz.
\end{split}
\end{align}
Notice that
$\|\widetilde{T}^z\|_{L^p(\Om)\rightarrow L^p(\Om)}=1$ implies
$$\|\widetilde{T}^z\otimes id\|_{L^p(\Om, V_q(t\in (0,R)))\rightarrow L^p(\Om, V_q(t\in (0,R)))}=1.$$
Therefore the last term in \eqref{8.1} is equal to
\begin{align*}
&\frac{1}{J^d}\int_{|z|\le J} \Big\| \frac{1}{|G_t|}\int_{G_t}[f_1(T^{x-z}\cdot), f_2(T^{y-z}\cdot)](\om)\,dx\, dy \|_{L^p(\Om,V_q(t\in(0,R)))}^p\, dz\\
&=\frac{1}{J^d}\int_{|z|\le J} \Big\|\mathcal{A}_t[f_1(T^{-z}\cdot), f_2(T^{-z}\cdot)](\om) \Big\|_{L^p(\Om,V_q(t\in(0,R)))}^p\, dz\\
&=\frac{1}{J^d}\int_{\Om}\|\A_t^G [F_1(\cdot,\om), F_2(\cdot,\om)](z) \|_{L^p(\R^d,V_q(t\in (0,R)))}^p \, d\om,
\end{align*}
where $F_1(z,\om)=f_1(T^{-z}\om)\1_{|z|\le J+R}$ and $F_2(z,\om)=f_2(T^{-z}\om)\1_{|z|\le J+R}$.

Applying Theorem \ref{thm:main1}, the last term above is bounded by
\begin{align*}
&\frac{1}{J^d}\int_{\Om}\| F_1(\cdot,\om)\|_{L^{p_1}(\R^d)}^p \|F_2(\cdot,\om) \|_{L^{p_2}(\R^d)}^p\, d\om\\
&\le \frac{1}{J^d} \left(\int_{\Om}\int_{\R^d} |F_1(z,\om)|^{p_1} \,dz \, d\om\right)^{\frac{p}{p_1}}\left(\int_{\Om}\int_{\R^d} |F_2(z,\om)|^{p_2} \,dz\, d\om\right)^{\frac{p}{p_2}}\\
&\le \frac{1}{J^d} \left(\int_{|z|\le J+R} \int_{\Om}|f_1(T^{-z}\om)|^{p_1} \, d\om\, dz\right)^{\frac{p}{p_1}}\left(\int_{|z|\le J+R} \int_{\Om}|f_2(T^{-z}\om)|^{p_2} \, d\om\, dz\right)^{\frac{p}{p_2}}\\
&\lesssim \frac{(J+R)^d}{J^d} \|f_1\|_{L^{p_2}(\Om)}^p \|f_2\|_{L^{p_2}(\Om)}^p.
\end{align*}
Since the estimates holds for all $J>0$, letting $J\rightarrow +\infty,$ we obtain
\begin{align*}
\frac{1}{J^d}\int_{\Om}\| F_1(\cdot,\om)\|_{L^{p_1}(\R^d)}^p \|F_2(\cdot,\om) \|_{L^{p_2}(\R^d)}^p\, d\om \lesssim \|f_1\|_{L^{p_2}(\Om)} ^p \|f_2\|_{L^{p_2}(\Om)}^p.
\end{align*}
Putting all together we have for any $R>0,$
\begin{align*}
&\left\|\Big\|\frac{1}{|G_t|}\int_{G_t}f_1(T^x\om) f_2(T^y\om)\, dx\, dy\Big\|_{V_q(t\in (0,R))}\right\|_{L^p(\Om)}^p \lesssim \|f_1\|_{L^{p_2}(\Om)}^p \|f_2\|_{L^{p_2}(\Om)}^p.
\end{align*}
Letting $R\rightarrow +\infty$ and applying  monotone convergence theorem, Corollary \ref{cor7.2} follows.
\end{proof}

\noindent \textbf{Acknowledgement.} The first author is partially supported by the NSF of China(Nos. 11571160, 11871096), the second author partially supported by the NSF of China(Nos. 11601396, 11431011), and the third author partially supported by NSF of China(No. 11671397).


\end{document}